\documentclass[a4paper,11pt]{amsart}

\usepackage[english]{babel}    % babel macros
\usepackage[latin1]{inputenc}  % Unix character set
\usepackage{amsfonts}          % AMS packages
\usepackage{amsmath}
\usepackage{amssymb}
\usepackage{amsthm}
\usepackage{amsxtra}
\usepackage{graphicx}          % for pictures
\usepackage{url}               % URL formatting
\usepackage{booktabs}          % for table formatting

\DeclareMathOperator{\aut}{Aut}
\newcommand{\F}{\mathbb{F}}
\newcommand{\symm}{\mathrm{Sym}}
\newcommand{\alt}{\mathrm{Alt}}

\newcommand{\atlas}{Atlas}
\newcommand{\twistA}{{^2\!A}}
\newcommand{\twistB}{{^2\!B}}
\newcommand{\twistD}{{^2 D}}
\newcommand{\trialD}{{^3 D}}
\newcommand{\twistE}{{^2\!E}}

\newcommand{\twistF}{{^2\!F}}
\newcommand{\twistG}{{^2 G}}
\newcommand{\GL}{\mathrm{GL}}
\newcommand{\SL}{\mathrm{SL}}
\newcommand{\PSL}{\mathrm{PSL}}
\newcommand{\SU}{\mathrm{SU}}
\newcommand{\PSU}{\mathrm{PSU}}
\newcommand{\Sp}{\mathrm{Sp}}
\newcommand{\exctypes}{\mathrm{Exc}}
\newcommand{\form}{\mathbf{f}}
\newcommand{\PGamma}{\mathrm{P}\Gamma}
\newcommand{\POmega}{\mathrm{P}\Omega}
\newcommand{\PDelta}{\mathrm{P}\Delta}
\newcommand{\geomclass}{\mathcal{C}}
\newcommand{\sclass}{\mathcal{S}}
\newcommand{\altcover}{\mathrm{\tilde{A}lt}}

\theoremstyle{plain}
\newtheorem{thm}{Theorem}[section]
\newtheorem{lemma}[thm]{Lemma}
\newtheorem{cor}[thm]{Corollary}

\theoremstyle{definition}

\theoremstyle{remark}
\newtheorem*{remark}{Remark}

\newcommand{\otoprule}{\midrule[\heavyrulewidth]}

\begin{document}

\title[Cross-characteristic representation growth]{Growth of cross-characteristic representations of finite quasisimple groups of Lie type}
\author{Jokke H\"as\"a}
\address{Department of Mathematics, Imperial College London, South Kensington Campus, London SW7 2AZ, United Kingdom}
\email{jokke.hasa@helsinki.fi}

\begin{abstract}
In this paper we give a bound to the number of conjugacy classes of maximal subgroups of any almost simple group whose socle is a classical group of Lie type. The bound is $2n^{5.2}+n\log_2\log_2 q$, where $n$ is the dimension of the classical socle and $q$ is the size of the defining field. To obtain the bound, we first bound the number of projective cross-characteristic representations of simple groups of Lie type as a function of the representation degree. These bounds are computed for different families of groups separately. In the computation, we use information on class numbers, minimal character degrees and gaps between character degrees.
\end{abstract}

%\begin{keyword}
%representation growth \sep maximal subgroups \sep classical groups \MSC{20C33}
%\end{keyword}

\maketitle

\section{Introduction}

Let $H$ be a finite quasisimple group, with $H/Z(H)$ a simple group of Lie type defined over a field of characteristic $p$. We are interested in the number of inequivalent $n$-di\-men\-sion\-al irreducible modular representations of $H$. In \cite{GuralnickLarsenTiep}, R.~Guralnick, M.~Larsen and P.~H.~Tiep obtain an upper bound of $n^{3.8}$ for the number of irreducible representations with dimension at most $n$ in the defining characteristic $p$. They use their result to find an asymptotic bound for the number $m(G)$ of conjugacy classes of maximal subgroups of an almost simple group $G$ with socle a group of Lie type. This bound is given as
\begin{equation*}
m(G)<ar^6+br\log\log q,
\end{equation*}
where $a$ and $b$ are unknown constants, and $r$ and $q$ are the rank of the socle and the size of its defining field, respectively.

In this paper, we sharpen the mentioned result in the case where the socle is a classical group, as follows:

\begin{thm}\label{theorem:main_theorem} Assume $G$ is a finite almost simple group with socle a classical group of dimension $n$ over the field $\F_q$. Let $m(G)$ denote the number of conjugacy classes of maximal subgroups of $G$ not containing the socle. Then
\begin{equation*}
m(G)<2n^{5.2}+n\log_2\log_2 q.
\end{equation*}
\end{thm}

Note that there are no unknown constants left in Theorem \ref{theorem:main_theorem}. To prove the result, we need to study the representation growth of groups of Lie type over fields of characteristic different from $p$. Define $r_n(H,\ell)$ as the number of inequivalent irreducible $n$-dimensional representations of $H$ over the algebraic closure of the finite prime field of characteristic $\ell\not=p$. Also, write $r_n^f(H,\ell)$ for the number of such said representations that are in addition faithful. If $\mathcal{L}$ is a family of finite quasisimple groups of Lie type,
we denote
\begin{equation*}
s_n(\mathcal{L},\ell)=\sum_{H\in\mathcal{L}}r_n^f(H,\ell).
\end{equation*}
The sum is taken over a set of representatives $H$ of isomorphism classes of quasisimple groups belonging to $\mathcal{L}$.
We will present upper bounds for the growth of $s_n(\mathcal{L},\ell)$ for different family of groups of Lie type. The upper bounds will have no dependence on $\ell$.

Regarding classical groups, we concern ourselves with the following families of quasisimple groups:
\begin{center}
\begin{tabular}{ll}
$A_1$ & linear groups in dimension 2 (but see below) \\
$A'$ & linear groups in dimension at least 3 \\
$\twistA$ & unitary groups in dimension at least 3 \\
$B$ & orthogonal groups in odd dimension $\ge 7$ over a field of odd size \\
$C$ & symplectic groups in dimension at least 4 \\
$D$ & orthogonal groups of plus type in even dimension $\ge 8$ \\
$\twistD$ & orthogonal groups of minus type in even dimension $\ge 8$.
\end{tabular}
\end{center}
From family $A_1$, we also exclude the linear groups $\PSL_2(4)\cong\PSL_2(5)\cong\alt(5)$ and $\PSL_2(9)\cong\alt(6)$, as well as all their covering groups.

\begin{thm}\label{theorem:classical_group_A1} For all $n>1$ and for any prime $\ell$, we have
\begin{equation*}
s_n(A_1,\ell)\le n+3.
\end{equation*}
\end{thm}

\begin{remark}
For small $n$, the maximal value of $s_n(A_1,\ell)/n$ over $\ell$ can be computed from the known decomposition tables. It is then possible to state the largest obtainable values of this maximum over all $n$. The three largest ones are $7/6$, $13/12$ and $16/15$. The largest appears for $n=6$ and $n=12$, and is used in the next theorem.
\end{remark}

\begin{thm}\label{theorem:classical_groups} Let $\mathcal{L}$ denote one of the above families of finite quasisimple classical groups. For all $n>1$ and for any prime $\ell$, we have
\begin{equation*}
s_n(\mathcal{L},\ell)\le c_\mathcal{L}\,n
\end{equation*}
where the constants $c_{\mathcal{L}}$ are shown in Table \ref{table:bounds_classical}.

\begin{table}[hbt]
\centering
\[\begin{array}{rcccccccc}
\toprule
\mathcal{L}: &A_1 & A' & \twistA & B & C & D & \twistD \\
\midrule
c_\mathcal{L}: & 7/6 & 1.5484 & 2.8783 & 0.9859 & 2.8750 & 1.5135 & 1.7969 \\
\bottomrule
\end{array}\]
\caption{Bounding constants for classical groups}\label{table:bounds_classical}
\end{table}
\end{thm}

Bounds in the last theorem are probably far from optimal (except for $A_1$), as can be judged from the known maximal values of $s_n(\mathcal{L},\ell)/n$ for $n\le 250$ (see Table~\ref{table:maximal_Qn_for_small_n} on page \pageref{table:maximal_Qn_for_small_n}).

The next theorem deals with the exceptional types. Let $\mathcal{E}$ denote the family of finite quasisimple groups with simple quotient an exceptional group of Lie type, excluding the groups $G_2(2)'$ and $\twistG_2(3)'$. (These are isomorphic to $\SU_3(3)$ and $\SL_2(8)$, respectively.)

\begin{thm}\label{theorem:exceptional_groups} For all $n>1$ and for any prime $\ell$, we have
\begin{equation*}
s_n(\mathcal{E},\ell)<1.2795n.
\end{equation*}
\end{thm}

It is straightforward to add together all constants pertaining to different families to get a bound for the representation growth of groups of Lie type. Let $\mathcal{Q}$ denote the class of all finite quasisimple groups of Lie type, still excluding $\PSL_2(4)\cong\PSL_2(5)$ and $\PSL_2(9)$ (these are counted as alternating groups).

\begin{cor}\label{corollary:final_constant} For all $n>1$ and for any prime $\ell$, we have
\begin{equation*}
s_n(\mathcal{Q},\ell)<14.1n.
\end{equation*}
\end{cor}

In Section \ref{section:preliminaries} we discuss the results from the literature that we will need for proving the theorems related to representation growth (\ref{theorem:classical_group_A1}, \ref{theorem:classical_groups} and \ref{theorem:exceptional_groups}). Section \ref{section:classical_groups} is devoted to proving Theorems \ref{theorem:classical_group_A1} and \ref{theorem:classical_groups}, and Section~\ref{section:exceptional_groups} to proving Theorem \ref{theorem:exceptional_groups}. The Main Theorem \ref{theorem:main_theorem} is proved in Section~\ref{section:main_theorem}.

\section{Preliminaries}\label{section:preliminaries}

Let $H$ be a finite quasisimple group of Lie type defined over a field $\F_q$. Then $H$ is a factor group of the universal covering group of the simple group $H/Z(H)$. All linear representations of $H$ are included in the representations of the universal covering group, so we may restrict our attention to this group.

The universal covering group is completely determined by its simple quotient, which in turn is determined by its Lie family, rank $r$ and the field size~$q$. We will denote the universal covering group by $H_r(q)$, where $H$ is replaced by the letter of the Lie family in question. For example, $A_2(3)$ is the group $\SL_3(3)$, and $E_6(4)$ is the triple cover of the finite simple group of Lie type $E_6$ over the field of four elements. Notice that this notation differs from the one used in the \atlas{} of Finite Groups (\cite{Atlas}).

Apart from finitely many exceptions, the universal covering group of a simple group of Lie type is obtained as a fixed-point group of a Frobenius morphism of a simply-connected simple algebraic group of the same type. The exceptions (given e.g.\ in \cite[Th.~5.1.4]{BlueBook}) are $A_1(4)$, $A_1(9)$, $A_2(2)$, $A_2(4)$, $A_3(2)$, $\twistA_3(2)$, $\twistA_3(3)$, $\twistA_5(2)$, $B_3(3)$, $C_2(2)$, $C_3(2)$, $D_4(2)$, $\twistE_6(2)$, $F_4(2)$, $G_2(3)$, $G_2(4)$, $\twistB_2(8)$ and $\twistF_4(2)$ (the Tits group).

The proofs below rely mainly on two types of information: results on smallest possible representation degrees of the universal covering groups, and bounds for their conjugacy class numbers. On the first topic, there is a good survey article by P.~H.~Tiep \cite{TiepMinimal}.

Landazuri and Seitz found lower bounds for the representation degrees in \cite{LandazuriMinimal}, and the bounds were subsequently improved by Seitz and Zalesskii~\cite{SeitzZalesskiiMinimal}. These bounds were given for each Lie type as functions of rank and the size of the defining field. The bounds have later been slightly improved by various authors.

It is generally the case that the group $H$ has a few representations of the smallest degree $n_1$, maybe a few also of the degrees $n_1+1$ and $n_1+2$, after which there is a relatively large ``gap'' before the next degrees. Then, after a couple of degrees, there is again a gap before the next one, and so on. The size of the first gap is known for linear, unitary and symplectic groups, and the most recent results can be found in \cite{GuralnickTiepLinear}, \cite{HissMalleUnitary}, \cite{GMSTUnitarySymplectic} and \cite{GuralnickTiepSymplecticEven}. We will constantly refer to degrees \emph{below} and \emph{above the gap} when we talk about these groups, but it will be made clear in the context which degrees belong to which class. For the other classical groups, we take the first gap to exist before the smallest degree, so that with these groups, the phrase ``degrees below the gap'' comes to mean the empty set.

For example, any group $A_r(q)=\SL_{r+1}(q)$ with $r\ge 2$ has irreducible cross-characteristic representations of dimensions
\begin{align*}
& \frac{q^{r+1}-q}{q-1}-\kappa_{r,q,\ell}\ =\ q^r+q^{r-1}+\cdots+q-\kappa_{r,q,\ell} \\
\text{and} \quad & \frac{q^{r+1}-1}{q-1} \hphantom{\mbox{}-\kappa_{r,q,l}}\ =\ q^r+q^{r-1}+\cdots+q+1,
\end{align*}
where $\kappa_{r,q,\ell}$ is either 0 or 1, depending on $r$, $q$ and $\ell$. The difference between these two dimensions is at most two, and they are both said to be ``below the gap''. The next dimension is ``above the gap'', given by a polynomial with degree at least $2r-3$.

Upper bounds for the number of conjugacy classes of classical groups were found by J.~Fulman and R.~Guralnick in \cite{GuralnickConjugacy}. These bounds are all of the form $k(H)\le q^r+B_\mathcal{L}q^{r-1}$, where $B_\mathcal{L}$ is a constant depending on the classical family $\mathcal{L}$ to which $H$ belongs.
The values for $B_\mathcal{L}$ are given in Table \ref{table:class_number_bounds}. For the exceptional groups, as well as classical groups of rank less than 8, one can obtain the precise conjugacy class numbers from Frank Lübeck's online data (\cite{LuebeckPolys}), where he lists all complex character degrees and their multiplicities for many kinds of groups of Lie type with rank at most eight.

\begin{table}[hbt]
\centering
\[
\begin{array}{rcccccc}
\toprule
\mathcal{L}: & A & \twistA & B & C & D & \twistD \\
\midrule
B_\mathcal{L}: & 3 & 15 & 22 & 30 & 32 & 32 \\
\bottomrule
\end{array}
\]
\caption{Fulman--Guralnick bounds for class numbers of classical groups}\label{table:class_number_bounds}
\end{table}

Finally, G.~Hiss and G.~Malle have determined all cross-characteristic representations of quasisimple groups with degree at most 250 (\cite{HissMalle250,HissMalle250Corrigenda}). This enables us to deal with the small degrees separately and assume in the general case that the degree is greater than 250. The Atlas of Finite Groups \cite{Atlas} together with the Atlas of Brauer Characters \cite{ModularAtlas} tell us which groups have all representation degrees below 250, so that these groups can be discarded from consideration. Also the groups with exceptional covering group can mostly be handled with the help of the \atlas es.

For technical reasons, we need to handle the types $A_1$, $\twistB_2$ and $\twistG_2$ separately. The generic complex character tables for these groups are known, and the cross-characteristic decomposition matrices can be found in \cite{BurkhardtPSL}, \cite{BurkhardtSuz}, \cite{HissHabilitation} and \cite{LandrockMichler}.

\section{Classical groups}\label{section:classical_groups}

In this section, we prove Theorems \ref{theorem:classical_group_A1} and \ref{theorem:classical_groups}. First we organise the results on gaps and minimal character degrees available and set up some lemmata for later use. We take $\ell$, the characteristic of the representation space, to be fixed, and suppress the notation as $r_n(H,\ell)=r_n(H)$.

Assume that $H=H_r(q)$ is a universal covering group of a classical simple group of rank $r>1$, defined over $\F_q$, where the characteristic of $\F_q$ is not $\ell$. Suppose also, unless otherwise mentioned, that $H$ is not an exceptional cover, but of simply-connected type. In the general treatment, we also assume for convenience that $H$ is not $A_2(3)$, $A_2(5)$, $A_3(3)$, $A_5(2)$, $A_5(3)$, $\twistA_2(3)$, $\twistA_2(4)$, $\twistA_2(5)$, $C_2(3)$ or $\twistD_4(2)$.
To get upper bounds for those ranks and field sizes that may yield representations of degree $n$, we need to bound the possible representation degrees of $H_r(q)$ from below.

Firstly, there is a small number of small degrees of $H$ below the gap. For fixed rank, these dimensions are given by some polynomials in $q$, whose degrees depend on the rank. The polynomials are listed in various sources and collected in Table \ref{table:classical_smallest_dimensions}, together with multiplicities.

\begin{table}[hbt]
\[
\begin{array}{cccc}
\toprule
\text{group} & \text{degree} & \text{multiplicity} & \text{reference} \\
\otoprule
A_r(q),\,r>1 & \frac{q^{r+1}-q}{q-1}-\genfrac{\{}{\}}{0pt}{}{0}{1} & 1 & \text{\cite{GuralnickTiepLinear}} \\
 & \frac{q^{r+1}-1}{q-1} & (q-1)_{l'}-1 & \\
\midrule
\twistA_r(q) & \frac{q^{r+1}-q(-1)^r}{q+1} & 1 & \text{\cite{HissMalleUnitary}} \\
 & \frac{q^{r+1}+(-1)^r}{q+1} & (q+1)_{l'}-1 & \\
\midrule
C_r(q),\,\text{$q$ odd} & \frac{1}{2}(q^r-1) & 2 & \text{\cite{GMSTUnitarySymplectic}} \\
 & \frac{1}{2}(q^r+1) & 2 & \\
\bottomrule
\end{array}
\]
\caption{Minimal representation degrees of some classical groups. The symbol $\genfrac{\{}{\}}{0pt}{}{0}{1}$ means either 0 or 1, depending on $r$, $q$ and $\ell$.}\label{table:classical_smallest_dimensions}
\end{table}

For each classical type $\mathcal{L}$ in $\{A',\twistA,C\}$, it is straightforward to find a lower bound for all the degree polynomials given in Table \ref{table:classical_smallest_dimensions}. The bounds that will be used are given in Table \ref{table:classical_smallest_dimension_bounds} as $\varphi_{\mathcal{L},r}$. It is also easy to bound the multiplicities from above and we list some bounds in the same table as $\psi_\mathcal{L}$.

Now, if $n$ is a representation degree of one of the groups $H_r(q)$ mentioned in Table~\ref{table:classical_smallest_dimensions}, we know that $n\ge\varphi_{\mathcal{L},r}(q)$. For each $\mathcal{L}$, this gives an upper bound to $q$. For example, if $n$ is a representation degree of $A_r(q)$, then $n\ge q^r$, so $q\le n^{1/r}$. These upper bounds are given in Table~\ref{table:classical_smallest_dimension_bounds} as $q_{\mathcal{L},r}^1(n)$. Similarly, upper bounds can be found for the rank, and these are listed as~$r_{\mathcal{L}}^1(n)$.

\begin{table}[hbt]
\[
\begin{array}{ccccc}
\toprule
\mathcal{L} & \varphi_{\mathcal{L},r}(q) & \psi_\mathcal{L}(q) & q_{\mathcal{L},r}^1(n) & r_{\mathcal{L}}^1(n) \\
\otoprule
A' & q^r & q & n^{1/r} & \log_2 n \\
\midrule
\twistA & (q-1)^r,\,\frac{1}{2}q^r & q & n^{1/r}+1,\,(2n)^{1/r} & \log_2 n+1 \\
\midrule
C & \frac{1}{3}q^r & 2 & & \log_3 n+1 \\
\bottomrule
\end{array}
\]
\caption{Bounds for minimal representation degrees, their multiplicities and related quantities for some classical groups}\label{table:classical_smallest_dimension_bounds}
\end{table}

The following two lemmata will be used in bounding the number of representations with degree below the gap.

\begin{lemma} Write $H(q)$ for a universal covering groups of one of the types $A_r$, $\twistA_r$, or $C_r$, with fixed Lie family and rank. Let $f_1$ and $f_2$ be two polynomials expressing representation degrees of $H(q)$ below the gap, as given in the second column of  Table~\ref{table:classical_smallest_dimensions}. If $f_1(q_1)=f_2(q_2)$ for some prime powers $q_1$ and $q_2$, then $q_1=q_2$ and $f_1=f_2$.
\end{lemma}

\begin{proof}
Suppose first that $H(q)$ is of type $A_r$ with $r\ge 2$. We see from Table~\ref{table:classical_smallest_dimensions} that, for $i\in\{1,2\}$, the following holds:
\begin{equation*}
q^r<f_i(q)<(q+1)^r \quad \text{for all $q$}.
\end{equation*}
Hence, if $n=f_i(q)$, we must have $n^{1/r}-1<q<n^{1/r}$. There is, however, at most one integer $q$ that fits, so we must have $q_1=q_2$. Also, the difference between any two minimal degree polynomials for $A_r$ is uniformly either 1 or 2, so that $f_1(q)=f_2(q)$ only if $f_1=f_2$.

The cases for $\twistA_r$ and $C_r$ are handled similarly.
\end{proof}

\begin{lemma}\label{lemma:bound_below_gap} Let $H(q)$ be as in the previous lemma. Writing $r_n^<(H(q),\ell)$ for the number of irreducible $\ell$-modular representations of $H(q)$ with dimension $n$ below the gap, we have
\begin{equation*}
\sum_q r_n^<\bigl(H(q),\ell\bigr)<\psi_\mathcal{L}\bigl(q_{\mathcal{L},r}^1(n)\bigr),
\end{equation*}
with $\psi_\mathcal{L}$ and $q_{\mathcal{L},r}^1(n)$ as in Table \ref{table:classical_smallest_dimension_bounds}.
\end{lemma}

\begin{proof} By the previous lemma, for each $n$ there can be at most one value of~$q$, such that $n$ is a degree of an irreducible representation of $H(q)$ below the gap. Also, there can be at most one minimal degree polynomial $f$ such that $f(q)=n$. The multiplicity corresponding to $f$ is bounded from above by $\psi_\mathcal{L}(q)$. Furthermore, the polynomial $\psi_\mathcal{L}$ is non-decreasing, so an upper bound is obtained by considering an upper bound to $q$.
\end{proof}

After the few smallest representation degrees there is the first gap, and the next degrees are significantly larger. Lower bounds for these degrees that lie above the gap, is here called the \emph{gap bounds}, and they are again polynomials in $q$ with degree depending on $r$. We list the relevant information on the gap bounds in Table~\ref{table:gap_bounds}. (The bounds might not hold for some groups that were excluded in the beginning of this section.) Notice that for the groups not appearing in Table~\ref{table:classical_smallest_dimensions}, we take the gap bound to be the Landazuri--Seitz--Zalesskii bound (or one of its refinements) for the smallest dimension of a non-trivial irreducible representation.

\begin{table}[hbt]
\[
\begin{array}{lccc}
\toprule
\text{group} & \text{gap bound} & \text{remarks} & \text{ref.} \\
\otoprule
A'_r(q) & (q^r-1)\left(\frac{q^{r-1}-q}{q-1}-\genfrac{\{}{\}}{0pt}{}{0}{1}\right) & r\ge 4 & \text{\cite{GuralnickTiepLinear}} \\
 & (q-1)(q^2-1)/\gcd(3,q-1) & r=2 & \\
 & (q-1)(q^3-1)/\gcd(2,q-1) & r=3 & \\
\midrule
\twistA_r(q) & \frac{(q^{r+1}+1)(q^r-q^2)}{(q^2-1)(q+1)}-1 & \text{$r$ even} & \text{\cite{GMSTUnitarySymplectic}} \\ \addlinespace[2pt]
r\ge 4 & \frac{(q^{r+1}-1)(q^r-q)}{(q^2-1)(q+1)} & \text{$r$ odd} & \\
\midrule
\twistA_r(q) & \frac{1}{6}(q-1)(q^2+3q+2) & r=2,\,3\mid(q+1) & \text{\cite{HissMalleUnitary}} \\ \addlinespace[2pt]
r<4 & \frac{1}{3}(2q^3-q^2+2q-3) & r=2,\,3\nmid(q+1) & \\ \addlinespace[2pt]
 & \frac{(q^2+1)(q^2-q+1)}{\gcd(2,q-1)}-1 & r=3 & \\
\midrule
B_r(q) & \frac{(q^r-1)(q^r-q)}{q^2-1} & q=3,\,r\ge 4 & \text{\cite{HoffmanMinimal}} \\ \addlinespace[2pt]
 & \frac{q^{2r}-1}{q^2-1}-2 & q>3 &  \\
\midrule
C_r(q) & \frac{(q^r-1)(q^r-q)}{2(q+1)} & & \text{\cite{GMSTUnitarySymplectic, GuralnickTiepSymplecticEven}} \\
\midrule
D_r(q) & \frac{(q^r-1)(q^{r-1}-1)}{q^2-1} & q\le 3 & \text{\cite{HoffmanMinimal}} \\ \addlinespace[2pt]
 & \frac{(q^r-1)(q^{r-1}+q)}{q^2-1}-2 & q>3 & \\
\midrule
\twistD_r(q) & \frac{(q^r+1)(q^{r-1}-q)}{q^2-1}-1 & r\ge 6 & \text{\cite{HoffmanMinimal}} \\
 & 1026 & r=4,\,q=4 & \\
 & 151 & r=5,\,q=2 & \\
 & 2376 & r=5,\,q=3 & \\
\bottomrule
\end{array}
\]
\caption{Gap bounds for the classical groups}\label{table:gap_bounds}
\end{table}

In Table \ref{table:gap_bound_estimates}, we list as $\Gamma_{\mathcal{L},r}(q)$ some lower bounds for the gap bounds. As with degrees below the gap, these bounds yield upper bounds for those values of $q$, for which $n$ may be a representation degree above the gap. The bounds for $q$ are listed as $q_{\mathcal{L},r}^2(n)$. Similarly, we get upper bounds for the ranks and list those as $r_\mathcal{L}^2(n)$.

\begin{table}[hbt]
\[
\begin{array}{lcccc}
\toprule
\mathcal{L} & \Gamma_{\mathcal{L},r}(q) & q_{\mathcal{L},r}^2(n) & r_\mathcal{L}^2(n) & \text{remarks} \\
\otoprule
A' & q^{2r-2} & n^{1/(2r-2)} & \frac{1}{2}\log_2 n+1 & r\ge 4 \\
 & \frac{1}{4}q^3 & (4n)^{1/3} & & r=2,\,q\ge 7 \\ \addlinespace[2pt]
 & \frac{1}{3}q^4 & (3n)^{1/4} & & r=3,\,q\ge 4 \\
\midrule
\twistA & \frac{1}{2}q^{2r-2} & (2n)^{1/(2r-2)} & \frac{1}{2}\log_2 n+\frac{3}{2} & r\ge 4 \\ \addlinespace[2pt]
 & \frac{1}{6}q^3 & (6n)^{1/3} & & r=2 \\ \addlinespace[2pt]
 & \frac{2}{5}q^4 & (5n/2)^{1/4} & & r=3,\,q\ge 4 \\
\midrule
B & q^{2r-2} & n^{1/(2r-2)} & \frac{1}{2}\log_3 n+1 & r\ge 3 \\
\midrule
C & \frac{1}{4}q^{2r-1} & (4n)^{1/(2r-1)} & \frac{1}{2}\log_2 n+\frac{3}{2} & \\
\midrule
D, \twistD & q^{2r-3} & n^{1/(2r-3)} & \frac{1}{2}\log_2 n+\frac{3}{2} & r\ge 4 \\
\bottomrule
\end{array}
\]
\caption{Bounds for the gap bounds and related quantities for classical groups}\label{table:gap_bound_estimates}
\end{table}

The following lemmata are used in bounding the number of representations with degree above the gap. The first one is a simple estimate relieving us from having to sum over all integers when we cannot determine which of them are prime powers.

\begin{lemma} Suppose $K$ is a strictly increasing function on the integers and $Q$ is a positive integer. For summing the values of $K(q)$ over prime powers $q$, we have the following estimate:
\begin{equation*}
\sum_{q\le Q}K(q)\ <\ K(Q)+\frac{1}{2}\sum_{i=3}^{Q-1}K(i)+\sum_{i=1}^{[\log_2 Q]}K(2^i).
\end{equation*}
Here $[x]$ denotes the integral part of $x$.
\end{lemma}

\begin{proof} Firstly, each $q$ can be either odd or a power of two. The binary powers are handle by the second sum on the right hand side.

For odd values of $q$, we have two possibilities. If $Q$ is even, the integers from 3 to $Q$ can be partitioned into pairs $(2l-1,2l)$, and we have $K(2l)>K(2l-1)$ for all $l$, since $K$ is strictly increasing. Thus,
\begin{equation*}
\sum_{\substack{q\le Q \\ \text{$q$ odd}}} K(q)
\le\sum_{l=2}^{Q/2} K(2l-1)
=\frac{1}{2}\sum_{l=2}^{Q/2} 2K(2l-1)
<\frac{1}{2}\sum_{i=3}^Q K(i).
\end{equation*}
On the other hand, if $Q$ is odd, we have by the same argument
\begin{equation*}
\sum_{\substack{q\le Q \\ \text{$q$ odd}}}K(q)<K(Q)+\frac{1}{2}\sum_{i=3}^{Q-1}K(i).
\end{equation*}
In both cases, we see that the desired inequality holds.
\end{proof}

\begin{lemma}\label{lemma:bound_above_gap} Keep the notation of the previous lemma.
Assume additionally that $K$ is a polynomial function, and denote $K(x)=\alpha_0+\alpha_1 x+\cdots+\alpha_s x^s$. Then we have
\begin{equation*}
\sum_{q\le Q} K(q)<K(Q)+\frac{1}{2}\int_0^Q K(x)\,dx
+\alpha_0\log_2 Q+\sum_{j=1}^s\frac{2^j \alpha_j Q^j}{2^j-1}.
\end{equation*}
\end{lemma}

\begin{proof} The result follows directly from the lemma above. The first two terms on the right hand side are obvious. For the last two, write $M=[\log_2 Q]$, and consider
\begin{equation*}
\sum_{i=1}^{M}K(2^i)=\sum_{i=1}^{M}\left(\alpha_0+\sum_{j=1}^s\alpha_j 2^{ij}\right)=\alpha_0 M+\sum_{j=1}^s\alpha_j \sum_{i=1}^{M}(2^j)^i.
\end{equation*}
It only remains to apply the formula for a geometric sum.
\end{proof}

We now embark upon finding the growth bounds for classical groups. For technical reasons, we will handle type $A_1$ separately in the end of the section.

\begin{proof}[Proof of Theorem \ref{theorem:classical_groups} (excluding $A_1$)] Fix a classical family $\mathcal{L}$, and let $H_r(q)$ denote the universal covering group of type $\mathcal{L}$ with rank $r$, defined over $\F_q$. We take $\ell$, the characteristic of the representation space, to be fixed, and suppress the notation by writing $r_n(H,\ell)=r_n(H)$ and $s_n(H,\ell)=s_n(H)$.

We need to find an upper bound for the ratio
\begin{equation*}
Q_\mathcal{L}(n)=\frac{1}{n}\sum_r\sum_q r_n(H_r(q)).\footnotemark
\end{equation*}
For each classical type $\mathcal{L}$, some small ranks and field sizes may be disregarded because of isomorphisms between the small groups. For example, for family $B$ we take the smallest applicable rank to be 3, since $B_1(q)\cong A_1(q)$ and $B_2(q)\cong C_2(q)$ for every $q$. Also, we will not consider groups of type $B$ in even characteristic, as $B_r(2^k)\cong C_r(2^k)$ for all $r$. Furthermore, the groups in $A_1$ are handled in a separate proof after this one. For each type $\mathcal{L}$, the smallest applicable rank and field size are denoted $r_0=r_0(\mathcal{L})$ and $q_0=q_0(\mathcal{L})$, respectively, and they are recalled in Table~\ref{table:smallest_r_and_q}. \footnotetext{In this proof, sums over $r$ and $q$ are to be taken over the positive integers and prime powers, respectively.% However, those pairs $(r,q)$ are skipped which correspond to groups that have been explicitly excluded from consideration.
}

\begin{table}[hbt]
\[
\begin{array}{rcccccc}
\toprule
\mathcal{L}: & A' & \twistA & B & C & D & \twistD \\
\midrule
r_0: & 2 & 2 & 3 & 2 & 4 & 4 \\
q_0: & 2 & 2 & 3 & 2 & 2 & 2 \\
\bottomrule
\end{array}
\]
\caption{Smallest applicable ranks and field sizes}\label{table:smallest_r_and_q}
\end{table}

Let us first compute the maxima of $Q_\mathcal{L}(n)$ for $n\le 250$. Here we use the tables of Hiss and Malle from \cite{HissMalle250Corrigenda}. The results are shown in Table~\ref{table:maximal_Qn_for_small_n}. As these values are all less than the corresponding bounds given in the statement of the theorem, we may henceforth assume that $n>250$.

\begin{table}[hbt]
\[
\begin{array}{cccccc}
\toprule
A' & \twistA & B & C & D & \twistD \\
\midrule
1/2 & 2/3 & 2/27 & 3/13 & 1/8 & 1/33 \\
\bottomrule
\end{array}
\]
\caption{Maximal values of $Q_\mathcal{L}(n)$ for $n\le 250$}\label{table:maximal_Qn_for_small_n}
\end{table}

The following is a general discussion of the structure of the proof, which will later be carried out case-by-case for each family. In the general discussion, we silently assume that the universal covering groups are of simply-connected type, and moreover, do not appear in the list of groups excluded in the beginning of the section.

We first consider those groups that may yield a representation degree $n$ below the gap. This only applies to families $A'$, $\twistA$ and $C$. We let $r_n^<(H)$ denote the number of irreducible representations of $H$ with dimension $n$ below the gap, and desire a bound for the ratio
\begin{equation*}
Q_\mathcal{L}^<(n)=\frac{1}{n}\sum_r\sum_q r_n^<(H_r(q)).
\end{equation*}

For a given rank $r$, let $f_r$ denote the polynomial giving the smallest representation degree as exhibited in Table \ref{table:classical_smallest_dimensions}. Now, if $f_r(q_0)>n$, there is no group of type $\mathcal{L}$ having rank $r$ that would have an irreducible representation of dimension $n$. On the other hand, if $f_r(q_0)\le n$, the number of such representations is at most $\psi(q^1_r(n))$, according to Lemma~\ref{lemma:bound_below_gap}. We obtain the following upper bound for $Q_\mathcal{L}^<(n)$:
\begin{equation}\label{equation:below_gap_exact}
R_\mathcal{L}^<(n)=\frac{1}{n}\sideset{}{'}\sum_{r\ge r_0}\psi(q^1_r(n)).
\end{equation}
The prime $(')$ stands for taking the summand that corresponds to $r$ into account only when $f_r(q_0)\le n$ (this makes the sum finite).

The values of the expression $R_\mathcal{L}^<(n)$ can be computed explicitly for all values up to a desired $n_0$. After $n_0$, notice that $\psi$ is linear and $q^1_r(n)$ has the form $(an)^{1/r}$. This means that the function $r\mapsto\psi(q^1_r(n))$ is decreasing and convex for all $n>n_0$, so we may use a trapezial estimate to bound the sum in \eqref{equation:below_gap_exact} from above (see Figure~\ref{trapeziumFigure}). This estimate is explained in the following.

The sum in~\eqref{equation:below_gap_exact} is taken up to the biggest rank $r$ for which $f_r(q_0)\le n$. Denote this rank $\hat{r}(n)$. Since $f_r(q_0)$ is strictly increasing in $r$, the value of $\hat{r}(n)$ is also increasing in $n$. So, for $n>n_0$, we are taking the sum at least up to $r=\hat{r}(n_0)$, denoted simply $R_0$. Write
\begin{equation*}
M_r(n)=\psi(q^1_r(n)),
\end{equation*}
and let $r^1(n)$ denote the upper bound for the rank given in Table~\ref{table:classical_smallest_dimension_bounds}. Setting up a trapezium as in Figure~\ref{trapeziumFigure}, we get the following upper bound to $R_\mathcal{L}^<(n)$:
\begin{equation}\label{equation:below_gap}
\bar{R}_\mathcal{L}^<(n)=\frac{1}{n}\left(M_{r_0}(n)+\frac{r^1(n)-r_0}{2}\bigl(M_{r_0+1}(n)+M_{R_0}(n)\bigr)\right).
\end{equation}
Observe that $r^1(n)$ has the form $\log n+b$, whence it follows that $\bar{R}_\mathcal{L}^<(n)$ will be decreasing after some $n$. We will take care to choose $n_0$ above this point for each $\mathcal{L}$.

\begin{figure}
\centering
\includegraphics{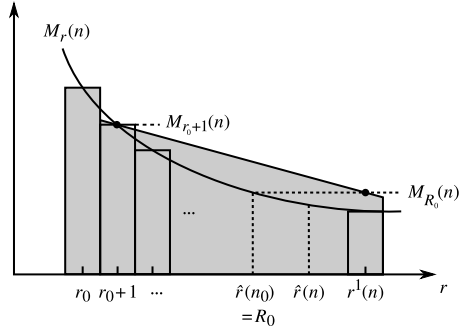}
\caption{The trapezium used to obtain \eqref{equation:below_gap}. Notice how the convexity of $M_r(n)$ guarantees that the shaded area is large enough to cover all the rectangles, even if $R_0=r^1(n)$.} \label{trapeziumFigure}
\end{figure}

Let us then write $r^>_n(H_r(q))$ for the number of irreducible representations of $H_r(q)$ with dimension $n$ greater than the gap bound. We are looking for a bound to
\begin{equation*}
Q_\mathcal{L}^>(n)=\frac{1}{n}\sum_r\sum_q r_n^>(H_r(q)).
\end{equation*}
This bound will consist of two terms: $R_\mathcal{L}^1(n)$ and $R_\mathcal{L}^2(n)$, corresponding to small and large ranks, respectively.

For small ranks, that is, with $r$ below some suitably chosen $r_1$, we can use precise conjugacy class numbers obtainable from Lübeck's online data (\cite{LuebeckPolys}) to bound $r_n^>(H_r(q))$ from above. Write $k_r(q)$ for the class number of $H_r(q)$, and define
\begin{equation}\label{equation:small_ranks_exact}
R_\mathcal{L}^1(n)=\frac{1}{n}\sum_{r=r_0}^{r_1-1}\sideset{}{'}\sum_{q\ge q_0}k_r(q).
\end{equation}
The inner (dashed) sum is taken over those $q$, for which the gap bound of $H_r(q)$ is at most $n$, and there are only finitely many such $q$.

The exact values of $R_\mathcal{L}^1(n)$ can be computed for $n$ up to any desired $n_0$. For larger values, we use a polynomial upper bound $K_r(q)=\sum_j\alpha_{r,j}q^j$ for the class number of $H_r(q)$. Although the class numbers are themselves given by polynomials in $q$, there can be finitely many different polynomials, each applicable to a certain congruence class of $q$. To obtain an upper bound, we simply choose the maximal one among these polynomials and leave out all negative terms. Then all the coefficients $\alpha_{r,j}$ will be non-negative integers, and it can be observed that the degree of $K_r(q)$ equals $r$.

Let $q_r^2(n)$ be the upper bound to $q$ as given in Table~\ref{table:classical_smallest_dimension_bounds}. The expression $R_\mathcal{L}^1(n)$ can now be bounded from above, in accordance with Lemma \ref{lemma:bound_above_gap}, by
\begin{multline}\label{equation:small_ranks}
\bar{R}_\mathcal{L}^1(n)=\frac{1}{n}\sum_{r=r_0}^{r_1-1}\left(K_r\bigl(q_r^2(n)\bigr)+\frac{1}{2}\int_0^{q_r^2(n)}K_r(x)\,dx \right. \\
\left. {}+\alpha_0\log_2 q_r^2(n)+\sum_{j\ge 1}\frac{\alpha_{r,j}\bigl(2q_r^2(n)\bigr)^j}{2^j-1}\right).
\end{multline}
The important thing to note here is that for all types $\mathcal{L}$ and all ranks $r$, the largest power of $n$ inside the outermost brackets is at most 1. (This can be seen by comparing the expression for $q_r^2(n)$ with the degree of $K_r(q)$, which is $r$.) With all coefficients $\alpha_{r,j}$ non-negative, we can then be sure that $\bar{R}_\mathcal{L}^1(n)$ is decreasing in $n$ when $\log_2 q_r^2(n)/n$ is. Since $q_r^2(n)$ is of the form $(an)^{1/d_r}$ for some $a$ and $d_r$, this happens after $n\ge e/a$.

For large ranks, that is, for $r\ge r_1$, we need to use the class number bounds of Fulman and Guralnick (\cite{GuralnickConjugacy}). With $B_\mathcal{L}$ given in Table \ref{table:class_number_bounds}, these bounds have the form $b_r(q)=q^r+B_\mathcal{L}q^{r-1}$.
We write
\begin{equation*}
R_\mathcal{L}^2(n)=\frac{1}{n}\sum_{r\ge r_1}
\sideset{}{'}\sum_{q\ge q_0}\left(q^r+B_\mathcal{L}q^{r-1}\right).
\end{equation*}
The latter sum is again taken over those $q$, for which the gap bound of $H_r(q)$ is at most $n$.

Also the values of $R_\mathcal{L}^2(n)$ can be computed exactly up to $n_0$. After that, we use an upper bound that results from substituting $b_r(q)$ in place of $K_r(q)$ in equation \eqref{equation:small_ranks}. When simplified, this becomes
\begin{equation*}
\frac{1}{n}\sum_{r=r_1}^{[r^2(n)]}\Lambda_r(n),
\end{equation*}
where
\begin{equation*}
\Lambda_r(n)=\frac{1}{2r+2}q_r^2(n)^{r+1}+\left(\frac{2^{r+1}-1}{2^r-1}+\frac{B_\mathcal{L}}{2r}\right)q_r^2(n)^r+\frac{(2^r-1)B_\mathcal{L}}{2^{r-1}-1}q_r^2(n)^{r-1}.
\end{equation*}

Notice that $q_r^2(n)$ has the form $(an)^{1/(br-c)}$ for some positive integers $b$ and~$c$. This makes every term in $\Lambda_r(n)$ decreasing in $r$ (for any fixed $n$), except possibly the final $(an)^\frac{r-1}{br-c}$, which is increasing if and only if $b>c$. In this exceptional case (which occurs only when $\mathcal{L}=C$), we replace $(an)^\frac{r-1}{br-c}$ by its limit $(an)^{1/b}$.\label{page:limit_for_C} Then we can use the simple estimate
\begin{equation}\label{equation:large_ranks}
\bar{R}_\mathcal{L}^2(n)=\frac{r^2(n)-r_1+1}{n}\,\Lambda_{r_1}(n)
\end{equation}
as an upper bound to $R_\mathcal{L}^2(n)$. Here, $r^2(n)$ is the upper bound for the rank as given in Table \ref{table:gap_bound_estimates}, and it has the form $a\log n+b$. Since the highest power of $n$ in $\Lambda_{r_1}(n)$ is less than 1, this means that $\bar{R}_\mathcal{L}^2(n)$ will become decreasing after some~$n$.

Finally, the ratio $Q_\mathcal{L}(n)$ is bounded above by the sum of $R_\mathcal{L}^<(n)$, $R_\mathcal{L}^1(n)$ and $R_\mathcal{L}^2(n)$. We shall finish the proof by computing the maximal values of these quantities for each classical family $\mathcal{L}$ separately. We also add the values arising from groups with exceptional covers individually in each case.

\medskip

\textbf{Case $\mathcal{L}=A'$.} Let us first assume that $(r,q)$ is none of $(2,2)$, $(2,3)$, $(2,4)$, $(2,5)$, $(3,2)$, $(3,3)$, $(5,2)$ or $(5,3)$. According to equation \eqref{equation:below_gap_exact} and Table \ref{table:classical_smallest_dimension_bounds}, we have
\begin{equation*}
R^<_{A'}(n)=\sideset{}{'}\sum_{r\ge r_0}\frac{1}{n^{1-1/r}},
\end{equation*}
where the summand corresponding to $r$ is taken into account if and only if $n\ge 2^{r+1}-3$ (except for $r=5$ where $q\ge 4$, so $n$ has to be at least $1363$).
The values of $R^<_{A'}(n)$ will be computed explicitly up to $n_0=650000$. At $n_0$, the biggest $r$ in the sum is 18, so looking at equation~\eqref{equation:below_gap}, we have
\begin{equation*}
R^<_{A'}(n)<\bar{R}^<_{A'}(n)=\frac{1}{n^{1/2}}+\left(\frac{1}{2}\log_2 n-1\right)\left(\frac{1}{n^{2/3}}+\frac{1}{n^{17/18}}\right)
\end{equation*}
for $n>n_0$. This expression is decreasing after $n_0$.

Let us then bound $Q^>_{A'}(n)$. We set $r_1=5$. For ranks less than $r_1$, we use the conjugacy class numbers obtainable from Lübeck's data to determine the values of $R^1_{A'}(n)$.
The values of $R^1_{A'}(n)$ are computed up to $n_0$. After this, we substitute $q_{A',r}^2(n)$ from Table \ref{table:gap_bound_estimates} into equation \eqref{equation:small_ranks} to get a decreasing expression $\bar{R}^1_{A'}(n)$ that can be used as an upper bound.

We move on to large ranks. Write $\Gamma_r(q)$ for the gap bound of $A_r(q)$. For rank 5, we have assumed that $q\ge 4$, and $\Gamma_5(4)=84909$. Thus, the smallest appearing gap bound is $\Gamma_6(2)=1827$. For values of $n$ from 1827 up to $n_0$, we will compute $R^2_{A'}(n)$, and afterwards we can substitute values from Tables~\ref{table:class_number_bounds} and \ref{table:gap_bound_estimates} into equation \eqref{equation:large_ranks} to obtain
\begin{equation*}
\bar{R}^2_{A'}(n)=\left(\frac{1}{2}\log_2 n-3\right)\left(\frac{1}{12n^{1/4}}+\frac{723}{310n^{3/8}}+\frac{31}{5n^{1/2}}\right).
\end{equation*}
This expression is decreasing for $n>n_0$.

Now, for $n$ from 251 up to $n_0$, we can compute the exact values of
\begin{equation*}
F(n)=R^<_{A'}(n)+R^1_{A'}(n)+R^2_{A'}(n),
\end{equation*}
and this bounds $Q_{A'}(n)$ from above. We are still, however, avoiding the exceptional cases. From the \atlas es, we find that the universal covering groups $A_2(2)$, $A_2(3)$, $A_2(4)$, $A_2(5)$ and $A_3(2)$ have no irreducible representations of dimension over 250, so we may forget about them. For $A_3(3)$, we extract the necessary data from the \atlas es for directly computing $r_n(A_3(3))$ for all~$n$. We find that the maximal value of $r_n(A_3(3))/n$ over all $n$ is less than 0.01924.

For the remaining two groups, we use the following information obtainable from \cite{GuralnickTiepLinear} and from the conjugacy class numbers of the groups, the latter of which are 60 for $A_5(2)$ and 396 for $A_5(3)$.

\begin{center}
\begin{tabular}{ccc}
\toprule
group & degree & multiplicity (at most) \\
\otoprule
$A_5(2)$ & $\ge 526$ & 57 \\
\midrule
$A_5(3)$ & 362\text{ or }363 & 1 \\
 & 364 & 1 \\
 & $\ge 6292$ & 393 \\
\bottomrule
\end{tabular}
\end{center}
\vspace{0.2cm}

Having added to $F(n)$ the values of $r_n(A_3(3))/n$ (for all $n$), $57/n$ for $n\ge 526$, $1/n$ for $n\in\{362,363,364\}$, and $393/n$ for $n\ge 6292$, we find that the biggest total sum below $n_0$ appears at $n=251$, and is less than 1.5484. More precisely, we have $R^<_{A'}(251)=0.115$, $R^1_{A'}(251)=1.435$, and $R^2_{A'}(251)=r_{251}(A_3(3))=0$.

For $n>n_0$, we use
\begin{equation*}
\bar{F}(n)=\bar{R}^<_{A'}(n)+\bar{R}^1_{A'}(n)+\bar{R}^2_{A'}(n)
\end{equation*}
as a decreasing upper bound to $F(n)$. We have $\bar{F}(n_0)<1.5267$, so even with the additions coming from the exceptional cases, we see that $F(n)<F(251)$ when $n>n_0$. Since according to Table~\ref{table:maximal_Qn_for_small_n}, we have $Q_{A'}(n)<F(251)$ also for $n\le 250$, we know that the obtained upper bound works for all $n$.

\textbf{Case $\mathcal{L}=\twistA$.} We assume that $(r,q)$ is none of $(2,2)$ (soluble), $(2,3)$, $(2,4)$, $(2,5)$, $(3,2)$, $(3,3)$ or $(5,2)$. Proceeding in the same way as in the case  $\mathcal{L}=A'$, we read from Table \ref{table:classical_smallest_dimension_bounds} that
\begin{equation*}
R^<_{2A}(n)=\sideset{}{'}\sum_{r\ge r_0}\left(\frac{1}{n^{1-1/r}}+\frac{1}{n}\right),
\end{equation*}
where each summand is taken into account when $n$ is at least
$(2^{r+1}-2)/3$, if $r$ is even, or $(2^{r+1}-1)/3$, if $r$ is odd.
The values of $R^<_{2A}(n)$ will be computed explicitly up to $n_0=80000$. At $n_0$, the biggest rank is 16, so we use
\begin{equation*}
\bar{R}_{2A}^<(n)=\frac{\sqrt{2}}{n^{1/2}}+(\log_2 n-1)\left(\frac{1}{(2n)^{2/3}}+\frac{1}{(2n)^{15/16}}\right).
\end{equation*}
(Here we applied the second value given for $q_{2A,r}^1(n)$ in Table \ref{table:classical_smallest_dimension_bounds}.) This function is decreasing for $n>n_0$.

To bound $Q^>_{2A}(n)$, we set $r_1=7$. For ranks less than $r_1$, we use Lübeck's class number polynomials to compute $R^1_{2A}(n)$ up to $n=n_0$.
After this, we use $\bar{R}^1_{2A}(n)$, defined in \eqref{equation:small_ranks}, as a decreasing upper bound for $R^1_{2A}(n)$.

For $r\ge r_1$, the smallest gap bound is 3570 (the corresponding group is $\twistA_7(2)$). For values of $n$ from 3570 up to $n_0$, we will compute $R^2_{2A}(n)$, and afterwards we will use
\begin{equation*}
\bar{R}^2_{2A}(n)=(\log_2 n-9)\left(\frac{1}{16\cdot(2n)^{1/3}}+\frac{5475}{1778\cdot(2n)^{5/12}}+\frac{635}{21\cdot(2n)^{1/2}}\right).
\end{equation*}
This is decreasing in $n$ for $n>n_0$.

For the exceptional cases, $\twistA_2(3)$, $\twistA_2(4)$, $\twistA_2(5)$ and $\twistA_3(2)$ have all degrees below 250, so this group can be ignored. The values of $r_n\bigl(\twistA_3(3)\bigr)/n$ can be computed using the \atlas es, and the maximum is 0.03572. Finally, group $\twistA_5(2)$ has 131 non-trivial conjugacy classes. Now, the values of
\begin{equation*}
F(n)=R^<_{2A}(n)+R^1_{2A}(n)+R^2_{2A}(n)+\frac{r_n\bigl(\twistA_3(3)\bigr)}{n}+\frac{131}{n}
\end{equation*}
are computed up to $n_0$, and the maximal value is $F(272)<2.8783$. (We have $R^<_{2A}(272)=0.17$, $R^1_{2A}(272)=2.24$ and \mbox{$R^2_{2A}(272)=r_{272}(\twistA_3(3))=0$}.) After $n_0$, the upper bound
\begin{equation*}
\bar{F}(n)=\bar{R}^<_{2A}(n)+\bar{R}^1_{2A}(n)+\bar{R}^2_{2A}(n)+0.03572+131/n
\end{equation*}
is decreasing, and $\bar{F}(n_0)<2.873$.

\textbf{Case $\mathcal{L}=B$.} For the orthogonal groups, we have no gap results. Assume that $(r,q)$ is not $(3,3)$. Also, we are assuming that $r\ge 3$ and $q\ge 3$, since otherwise we would be in coincidence with the symplectic groups.
We set $r_1=5$. For ranks less than $r_1$, we compute $R^1_B(n)$ up to $n_0=58000$.
When $n>n_0$, we use $\bar{R}^1_B(n)$ as a decreasing upper bound to $R^1_B(n)$.

For $r\ge r_1$, the smallest gap bound is 7260 (the corresponding group is $B_5(3)$). For values of $n$ from 7260 up to $n_0$, we can compute $R^2_B(n)$, and afterwards we use the upper bound
\begin{equation*}
\bar{R}^2_B(n)=\left(\frac{1}{2}\log_3 n-3\right)\left(\frac{1}{12n^{1/4}}+\frac{656}{155n^{3/8}}+\frac{682}{15n^{1/2}}\right).
\end{equation*}
This is decreasing for $n>n_0$.

The exceptional case $B_3(3)$ has 87 non-trivial conjugacy classes.
The values of
\begin{equation*}
F(n)=R^1_B(n)+R^2_B(n)+\frac{87}{n}
\end{equation*}
are computed up to $n_0$, and the maximal value is $F(780)<0.9859$. (We have $R^1_B(780)=0.88$ and $R^2_B(780)=0$.) After $n_0$, we apply
\begin{equation*}
\bar{F}(n)=\bar{R}^1_B(n)+\bar{R}^2_B(n)+\frac{87}{n},
\end{equation*}
which is decreasing, and has $\bar{F}(n_0)<F(780)$.

\textbf{Case $\mathcal{L}=C$.} Assume $(r,q)$ is none of $(2,2)$, $(2,3)$ or $(3,2)$. Below the gap bound given in Table \ref{table:gap_bounds}, representation degrees exist only for groups with odd $q$. Referring to Table \ref{table:classical_smallest_dimension_bounds}, we have
\begin{equation*}
R^<_C(n)=a_n\cdot\frac{2}{n},
\end{equation*}
where $a_n$ is the number of ranks $r\ge 2$ such that $n\ge(3^r-1)/2$.
The values of $R^<_C(n)$ will be computed up to $n_0=140000$. By using the bounds in Table~\ref{table:classical_smallest_dimension_bounds} to approximate $a_n\le\log_3(n)$, we obtain the upper bound
\begin{equation*}
\bar{R}_C^<(n)=\frac{2\log_3 n}{n}.
\end{equation*}
This is decreasing for $n>n_0$.

For $Q^>_C(n)$, we set $r_1=7$. For ranks less than $r_1$, we compute $R^1_C(n)$ up to $n_0$, and when $n>n_0$, we use $\bar{R}^1_C(n)$ as a decreasing upper bound to $R^1_C(n)$.

For $r\ge r_1$, the smallest gap bound is 2667 (the corresponding group is $C_7(2)$). For values of $n$ from 2667 up to $n_0$, we can compute $R^2_C(n)$, and afterwards we use the upper bound
\begin{equation*}
\bar{R}^2_C(n)=(\log_2 n-9)\left(\frac{1}{8\cdot(4n)^{5/13}}+\frac{7380}{889\cdot(4n)^{6/13}}+\frac{5080}{21n^{1/2}}\right).
\end{equation*}
(Note that the last term is chosen according to the discussion on page \pageref{page:limit_for_C}.) This bound is decreasing for $n>n_0$.

For the exceptional cases, $C_2(2)$ and $C_2(3)$ have all degrees below 250. The values of $r_n(C_3(2))/n$ can be computed from the \atlas es, and the maximum is 0.01072. The values of
\begin{equation*}
F(n)=R^<_C(n)+R^1_C(n)+R^2_C(n)+\frac{r_n(C_3(3))}{n}
\end{equation*}
are computed up to $n_0$, and the maximal value is $F(288)=2.8750$. (More precisely, $R^<_C(288)=1/36$, $R^1_C(288)=205/72$ and $R^2_C(288)=r_{288}(C_3(3))=0$.) After $n_0$, the upper bound $\bar{F}(n)=\bar{R}^<_C(n)+\bar{R}^1_C(n)+\bar{R}^2_C(n)$ is decreasing, and $\bar{F}(n_0)<F(288)$.

\textbf{Case $\mathcal{L}=D$.} Assume $(r,q)$ is not $(4,2)$. There are no gap results. We set $r_1=8$. For ranks less than $r_1$, we compute $R^1_D(n)$ up to $n_0=222000$. When $n>n_0$, we use $\bar{R}^1_D(n)$ as a decreasing upper bound to $R^1_D(n)$.

For $r\ge r_1$, the smallest gap bound is 11048 (the corresponding group is $D_8(2)$). For values of $n$ from 11048 up to $n_0$, we can compute $R^2_D(n)$, and afterwards we use the upper bound
\begin{equation*}
\bar{R}^2_D(n)=(\log_2 n-11)\left(\frac{1}{36n^{4/13}}+\frac{1021}{510n^{5/13}}+\frac{4080}{127n^{6/13}}\right).
\end{equation*}
This is decreasing for $n>n_0$.

The exceptional group $D_4(2)$ is covered in the modular \atlas, so we can compute the values of $r_n(D_4(2))/n$, the maximum of which is 0.005792. The values of
\begin{equation*}
F(n)=R^1_D(n)+R^2_D(n)+\frac{r_n(D_4(2))}{n}
\end{equation*}
are computed up to $n_0$, and the maximal value is $F(298)<1.5135$. (We have $R^1_D(298)=F(298)$ and $R^2_D(298)=r_{298}(D_4(2))=0$.) After $n_0$, the upper bound $\bar{F}(n)=\bar{R}^1_D(n)+\bar{R}^2_D(n)+0.005792$ is decreasing, and $\bar{F}(n_0)<F(298)$.

\textbf{Case $\mathcal{L}=\twistD$.} Assume $(r,q)$ is not $(4,2)$. There are no gap results. We set $r_1=7$. For ranks less than $r_1$, we compute $R^1_{2D}(n)$ up to $n_0=220000$. When $n>n_0$, we use $\bar{R}^1_{2D}(n)$ as a decreasing upper bound to $R^1_{2D}(n)$.

For $r\ge r_1$, the smallest gap bound is 2663 (the corresponding group is $\twistD_7(2)$). For values of $n$ from 2663 up to $n_0$, we can compute $R^2_{2D}(n)$, and afterwards we use the upper bound
\begin{equation*}
\bar{R}^2_{2D}(n)=(\log_2 n-9)\left(\frac{1}{32n^{3/11}}+\frac{3817}{1778n^{4/11}}+\frac{2032}{63n^{5/11}}\right).
\end{equation*}
This is decreasing for $n>n_0$.

The character degrees of $\twistD_4(2)$ are given in the \atlas es. The maximum of $r_n(\twistD_4(2))/n$ is 0.004202. The values of
\begin{equation*}
F(n)=R^1_{2D}(n)+R^2_{2D}(n)+\frac{r_n(\twistD_4(2))}{n}
\end{equation*}
are computed up to $n_0$, and the maximal value is $F(251)<1.7969$. (We have $R^1_{2D}(251)=F(251)$ and $R^2_{2D}(251)=r_{251}(\twistD_4(2))=0$.) After $n_0$, the upper bound $\bar{F}(n)=\bar{R}^1_{2D}(n)+\bar{R}^2_{2D}(n)+0.004202$ is decreasing, and $\bar{F}(n_0)<F(251)$.

\medskip

We have thus checked that all the bounds given in the statement of the theorem hold, except possibly for $\mathcal{L}=A_1$.
\end{proof}

We still need to check the case of rank one linear groups.

\begin{proof}[Proof of Theorem \ref{theorem:classical_group_A1}] We do not take groups $A_1(4)\cong A_1(5)$ or $A_1(9)$ into account, as they are isomorphic to alternating groups.

For $n\le 250$, the bounds given in the statement of the theorem hold according to the tables of Hiss and Malle (\cite{HissMalle250}). Moreover, it can be verified that the largest value of $s_n(A_1)/n$ for $n\le 250$ is $7/6$, and this can be obtained only at $n=6$ and $n=12$. Similarly, the second largest value $13/12$ can be obtained only at $n=24$, and the third largest value $16/15$ only at $n=30$.

We may now assume that $q>250$, since the maximal degree of an irreducible representation of $A_1(q)$ is $q+1$. With this assumption, the universal cover of $A_1(q)$ is $\SL_2(q)$. For this group, the complex character degrees and their multiplicities are listed in Table \ref{table:characters_SL2}.

\begin{table}[hbt]
\[
\begin{array}{ccccc}
\toprule
\multicolumn{2}{c}{\text{$q$ even}} & \quad & \multicolumn{2}{c}{\text{$q$ odd}} \\
\text{degree} & \text{multiplicity} && \text{degree} & \text{multiplicity} \\
\midrule
q-1 & q/2 && q-1 & (q-1)/2 \\
q & 1 && q & 3 \\
q+1 & (q-2)/2 && q+1 & (q-3)/2 \\
 & && (q-1)/2 & 2 \\
 & && (q+1)/2 & 2 \\
\bottomrule
\end{array}
\]
\caption{Non-trivial complex character degrees of $\SL_2(q)$}\label{table:characters_SL2}
\end{table}

The\label{proof_for_A1} cross-characteristic decomposition numbers for the simple groups of type $\PSL_2(q)$ are given in \cite{BurkhardtPSL}. If $q$ is even, $\SL_2(q)$ is isomorphic to $\PSL_2(q)$, and if $\ell$ is even, the scalar $-1\in\SL_2(q)$ acts trivially in the representation, so $\SL_2(q)$ has no faithful representations. Assume then that $q$ and $\ell$ are odd. When $q$ is not a power of $\ell$, the Sylow $\ell$-subgroups of $\SL_2(q)$ are cyclic. Now Dade's theorems can be used (see e.g.\ \cite[§68]{Dornhoff}), and in this case they tell us that the decomposition numbers are all either 0 or 1, there are at most two irreducible Brauer characters in each block, and the set of irreducible $\ell$-Brauer characters is a subset of the irreducible complex characters (restricted to $p$-regular elements). Therefore, it is enough to consider the case with $\ell=0$.

We can read off from the character table that for any degree $n$, there are at most $n+3$ characters of this degree, of any groups of type $\SL_2$. (Equality is obtained only if $n$, $n+1$, $n-1$, $2n+1$ and $2n-1$ are all prime powers.) Furthermore, when $n\ge 251$, we have $(n+3)/n<1.012<7/6$, so the upper bound obtained for $n\le 250$ holds globally. This proves Theorem~\ref{theorem:classical_group_A1}, and also completes the proof of Theorem~\ref{theorem:classical_groups} above.
\end{proof}

\section{Exceptional groups}\label{section:exceptional_groups}

In this section we shall prove Theorem \ref{theorem:exceptional_groups}. For universal covering groups $H(q)$ of exceptional Lie type, the rank is always bounded and there are no gap results. We will use Lübeck's conjugacy class numbers and the (sharpened) Landazuri--Seitz--Zalesskii bounds for the minimal representation degrees. The latter are listed in Table~\ref{table:exceptional_smallest_dimensions}, where $\Phi_k$ denotes the $k$'th cyclotomic polynomial in $q$. As in the previous section, these lower bounds lead to upper bounds for the values of $q$ for which $H(q)$ can have a representation of degree $n$. Such upper bounds that will be used later are listed in Table~\ref{table:exceptional_smallest_dimension_bounds} as $q_H^1(n)$.

\begin{table}[hbt]
\[
\begin{array}{lccc}
\toprule
\text{group} & \text{bound for rep.\ degree} & \text{remark} & \text{ref.} \\
\otoprule
\twistB_2(q) & (q-1)\sqrt{q/2} & & \text{\cite{LandazuriMinimal}} \\
\midrule
\trialD_4(q) & q^5-q^3+q-1 & & \text{\cite{MMT}} \\
\midrule
E_6(q) & q(q^4+1)(q^6+q^3+1)-1 & & \text{\cite{HoffmanTypeE}} \\
\midrule
\twistE_6(q) & (q^5+q)(q^6-q^3+1)-2 & & \text{\cite{MMT}} \\
\midrule
E_7(q) & q\Phi_7\Phi_{12}\Phi_{14}-2 & & \text{\cite{HoffmanTypeE}} \\
\midrule
E_8(q) & q\Phi_4^2\Phi_8\Phi_{12}\Phi_{20}\Phi_{24}-3 & & \text{\cite{HoffmanTypeE}} \\
\midrule
F_4(q) & \frac{1}{2}q^7(q^3-1)(q-1) & q\text{ even} & \text{\cite{LandazuriMinimal}} \\
 & q^6(q^2-1) & q\text{ odd} & \text{\cite{SeitzZalesskiiMinimal}} \\
\midrule
\twistF_4(q) & q^4(q-1)\sqrt{q/2} & & \text{\cite{LandazuriMinimal}} \\
\midrule
G_2(q) & q^2(q^2+1) & q\equiv 0\text{ (mod 3)} & \text{\cite{TiepMinimal}} \\
 & q^3 & q\equiv 1\text{ (mod 3)} & \\
 & q^3-1 & q\equiv 2\text{ (mod 3)} & \\
\midrule
\twistG_2(q) & q(q-1) & & \text{\cite{LandazuriMinimal}} \\
\bottomrule
\end{array}
\]
\caption{Lower bounds for representation degrees of the universal covering groups of exceptional groups of Lie type}\label{table:exceptional_smallest_dimensions}
\end{table}

\begin{table}[hbt]
\[
\begin{array}{ccc}
\toprule
H & q_H^1(n) & \text{remark} \\
\otoprule
\trialD_4 & (4/3\cdot n)^{1/5} & \\
E_6 & n^{1/11} & \\
\twistE_6 & (3/2\cdot n)^{1/11} & q\ge 3 \\
E_7 & n^{1/17} & \\
E_8 & n^{1/29} & \\
F_4 & (9/8\cdot n)^{1/8} & \\
\twistF_4 & (2n)^{2/11} & q\ge 8 \\
G_2 & (125/124\cdot n)^{1/3} & q\ge 5 \\
\bottomrule
\end{array}
\]
\caption{Upper bounds for such $q$ for which $H(q)$ of exceptional Lie type may have a representation of degree $n$. (See the text for details.)}\label{table:exceptional_smallest_dimension_bounds}
\end{table}

For technical reasons, we shall deal with the Suzuki and Ree types $\twistB_2$ and $\twistG_2$ separately. Write therefore $\exctypes'$ for the set of remaining exceptional Lie types (i.e., the set of symbols $\trialD_4$, $E_6$, $\twistE_6$, $E_7$, $E_8$, $F_4$, $\twistF_4$ and $G_2$).

Let $H(q)$ be a simply-connected universal covering group of a simple group of an exceptional Lie type $H\in\exctypes'$. Mimicking the previous section, we denote
\begin{equation*}
Q_\mathcal{E}(n)=\frac{1}{n}\sum_{H\in\exctypes'}\sum_q r_n\bigl(H(q)\bigr).
\end{equation*}

Write $f_H(q)$ for the lower bound of the smallest representation degree of $H(q)$ as given in Table \ref{table:exceptional_smallest_dimensions}. Now, $Q_\mathcal{E}(n)$ is bounded above by
\begin{equation*}
R_\mathcal{E}(n)=\frac{1}{n}\sum_{H\in\exctypes'}\sideset{}{'}\sum_{q}k_H(q),
\end{equation*}
where $k_H(q)$ is the conjugacy class number of $H(q)$, and the dashed sum is taken over those $q$ for which $H(q)$ is defined and $f_H(q)\le n$.

For large $n$, we shall need the estimate $R_\mathcal{E}(n)<\bar{R}_\mathcal{E}(n)$, where $\bar{R}_\mathcal{E}(n)$ is defined analogously to $\bar{R}_\mathcal{L}^1(n)$ on page \pageref{equation:small_ranks}, using $q^1_H(n)$ from Table~\ref{table:exceptional_smallest_dimension_bounds} instead of $q^2_{\mathcal{L},r}(n)$. The expression $\bar{R}_\mathcal{E}(n)$ becomes decreasing eventually. The main reason for this is that the ratio $\deg(k_H)/(\deg(f_H)+1)$ is at most 1, except for the two excluded types. (Note that for $H=\twistF_4$, $f_H$ is not actually a polynomial in $q$; in this case we say that $\deg(f_H)=11/2$.)

We will now establish the claimed bound for the representation growth of groups of exceptional Lie type. Let us first deal with the excluded types.

\begin{lemma}\label{lemma:suzuki_ree} For the Suzuki and Ree groups, if $n>250$, we have
\begin{equation*}
\sum_{q}r_n\bigl(\twistB_2(q)\bigr)<\sqrt{n},
\end{equation*}
and
\begin{equation*}
\sum_{q}r_n\bigl(\twistG_2(q)\bigr)<\sqrt[3]{n}.
\end{equation*}
\end{lemma}

\begin{proof} Firstly, we can confirm from the \atlas es that $\twistB_2(8)$ has all irreducible representation degrees below 250. Recall that we exclude $\twistG_2(2)$ and $\twistG_2(3)$ because they are isomorphic to classical groups. (Notice also that $\twistB_2(2)$ is soluble.) Thus we may assume that $q\ge 32$ for type $\twistB_2$ and $q\ge 27$ for $\twistG_2$. With these assumptions, the universal covering groups are isomorphic to the simple groups.

Let us begin with the groups $\twistB_2(q)$. The generic complex character table of $\twistB_2(q)$ was introduced in \cite{SuzukiGroups}, and the decomposition matrices for $\ell$-modular representations of these groups are given in \cite{BurkhardtSuz}. From this information, we see that the set of irreducible Brauer characters is a subset of the irreducible complex characters (restricted to $p$-regular elements), except when $\ell$ divides $q-\sqrt{2q}+1$. In the latter case, there appears an additional irreducible Brauer character with degree $q^2-1$. The possible $\ell$-modular degrees are all listed in Table~\ref{table:characters_2B2}. The multiplicities in the table are for complex characters (except for the degree $q^2-1$), as this gives an upper bound for them.

\begin{table}[hbt]
\[
\begin{array}{ccc}
\toprule
\text{degree} & \text{multiplicity} & \text{class} \\
\otoprule
(q-1)\sqrt{q/2} & 2 & \text{A} \\
\midrule
(q-\sqrt{2q}+1)(q-1) & (q+\sqrt{2q})/4 & \\
q^2-1 & 1 & \\
q^2 & 1 & \text{B} \\
q^2+1 & q/2-1 & \\
(q+\sqrt{2q}+1)(q-1) & (q-\sqrt{2q})/4 & \\
\bottomrule
\end{array}
\]
\caption{Non-trivial character degrees of ${}^2 B_2(q)$}\label{table:characters_2B2}
\end{table}

The character degrees of $\twistB_2(q)$ are given by functions of $q$ that are also polynomials in $\sqrt{q}$. These functions can be divided into two classes according to their polynomial degree, as shown in the last column of Table~\ref{table:characters_2B2}. The values of the functions in class B are increasing and strictly between $\frac{3}{4}q^2$ and $2q^2$. Suppose now that $f_1$ and $f_2$ are two functions from class B, and that $f_1(q_1)=f_2(q_2)$ for some $q_1<q_2$. As $\twistB_2(q)$ is only defined for odd powers of 2, we know that $q_2\ge 4q_1$. It follows that
\begin{equation*}
f_1(q_1)<2q_1^2<\frac{3}{4}\,q_2^2<f_2(q_2),
\end{equation*}
which is against the assumption. This means that if $f_1(q_1)=f_2(q_2)$ holds for two polynomials from class B, we must have $q_1=q_2$. However, it is easily checked that for any $q$, different polynomials in class B give different values. Hence it follows that if $n$ is a character degree of $\twistB_2(q)$, and $n\not=(q-1)\sqrt{q/2}$, then $q$ is fixed and $n=f(q)$ for a unique function in class B.

In the case just described (where $n$ is a value of a function from class B), we know that $n>\frac{3}{4}q^2$, so that $q<2\sqrt{n}/\sqrt{3}$. Substituting this to the largest multiplicity in class B and adding the 2 possible characters coming from class A, we find that
\begin{equation*}
\sum_{q\ge 32}r_n\bigl(\twistB_2(q)\bigr)<\frac{1}{2\sqrt{3}}\sqrt{n}+\frac{1}{2\sqrt[4]{3}}\sqrt[4]{n}+2.
\end{equation*}
This gives the first result.

In case of $\twistG_2$, a similar analysis can be made. The complex character degrees are given in \cite{WardReeGroups}. For $\ell=2$, the decomposition matrices can be found in \cite{LandrockMichler}, and for odd $\ell$, they can be inferred from the Brauer trees presented in \cite{HissHabilitation}. It turns out that in addition to the complex characters, there appears an irreducible Brauer character of degree $q^2-q$ if $\ell=2$, one of degree $q^3-1$ if $\ell$ divides $q^2-\sqrt{3}q+1$, and one of degree $(q-1)(q+2\sqrt{q/3}+1)(q-\sqrt{3q}+1)$ if $\ell=2$ or $\ell$ is odd and divides $q+1$. All these degrees and their multiplicities (or upper bounds to them) are presented in Table~\ref{table:characters_2G2_modular}.

\begin{table}[hbt]
\[
\begin{array}{ccc}
\toprule
\text{degree} & \text{multiplicity} & \text{class} \\
\otoprule
q^2-q & 1 & \text{A} \\
q^2-q+1 & 1 & \\
\midrule
\sqrt{q/3}(q-1)(q-\sqrt{3q}+1)/2 & 2 & \\
\sqrt{q/3}(q-1)(q+\sqrt{3q}+1)/2 & 2 & \text{B} \\
\sqrt{q/3}(q^2-1) & 2 & \\
\midrule
(q^2-1)(q-\sqrt{3q}+1) & (q+\sqrt{3q})/6 & \\
(q-1)(q+2\sqrt{q/3}+1)(q-\sqrt{3q}+1) & 1 & \\
(q-1)(q^2-q+1) & (q-3)/6 & \\
q(q^2-2+1) & 1 & \text{C} \\
q^3-1 & 1 & \\
q^3 & 1 & \\
q^3+1 & (q-3)/2 & \\
(q^2-1)(q+\sqrt{3q}+1) & (q-\sqrt{3q})/6 & \\
\bottomrule
\end{array}
\]
\caption[Non-trivial Brauer character degrees of $\twistG_2(q)$]{Non-trivial Brauer character degrees of $\twistG_2(q)$.}\label{table:characters_2G2_modular}
\end{table}

The degree functions are divided into three classes. We notice that in class C, the values of the functions are strictly between $\frac{2}{3}q^3$ and $2q^3$. Now, suppose that $f_1$ and $f_2$ are two functions from class C, such that $f_1(q_1)=f_2(q_2)$ for some $q_1<q_2$. As $q_2\ge 9q_1$, we get
\[
f_1(q_1)<2q_1^3\le\frac{2}{729}q_2^3<\frac{2}{3}q_2^3<f_2(q_2),
\]
a contradiction. As before, we can conclude that if the value $n$ is a character degree of $\twistG_2(q)$ given by a function in class C, then that function is unique, and so is $q$. Also, in this case we would have $n>\frac{2}{3}q^3$, so that $q<\sqrt[3]{3n}/\sqrt[3]{2}$. Substituting this upper bound into the largest multiplicity and adding the 3 degrees coming from classes A and B (also in these classes only one function can have the value $n$), we get the following estimate:
\[
\sum_{q\ge 27}r_n\bigl(\twistG_2(q)\bigr)
<\frac{\sqrt[3]{3}}{2\sqrt[3]{2}}\sqrt[3]{n}+\frac{3}{2}.
\]
This gives the second result.
%
% In case of $\twistG_2$, a similar analysis can be made, based on the information on complex characters given in \cite{WardReeGroups} and on decomposition matrices found in \cite{HissHabilitation} and \cite{LandrockMichler}. We omit the details.
\end{proof}

\begin{proof}[Proof of Theorem \ref{theorem:exceptional_groups}] Recall that we are not considering the groups $G_2(2)$ or $\twistG_2(3)$ here, as they are isomorphic to classical groups. From the tables of Hiss and Malle (\cite{HissMalle250}), we find that the stated bound holds for all $n\le 250$.

Assume that $n>250$. We first take care of the following exceptional cases: $\twistE_6(2)$, $F_4(2)$, $G_2(3)$, $G_2(4)$, and $\twistF_4(2)$. The universal covering groups $\twistF_4(2)$, $G_2(3)$ and $G_2(4)$ can be found in the \atlas es. Letting $\mathcal{X}$ denote the set of these four groups, we can compute
\begin{equation*}
\max_n\sum_{H\in\mathcal{X}}\frac{r_n(H)}{n}=0.03389.
\end{equation*}

For those groups that are not covered in the \atlas{} of Brauer Characters, $F_4(2)$ has 95 conjugacy classes. On the other hand, the Schur multiplier of the simple group of type $\twistE_6(2)$ is an abelian group of type $2^2\cdot 3$, and F.~Lübeck has computed the character table of the 6-fold subcover (\cite{Luebeck2E6}). From his data, we find that the subcover has 542 conjugacy classes, so $\twistE_6(2)$ has at most 1084. The smallest non-trivial representation degree of $\twistE_6(2)$ is at least 1536, according to~\cite{LandazuriMinimal}.

We use Lemma \ref{lemma:suzuki_ree} to estimate the contribution of types $\twistB_2$ and $\twistG_2$. Computing the values of
\begin{equation*}
F(n)=R_\mathcal{E}(n)+\frac{1}{n}\left(\sum_{H\in\mathcal{X}}r_n(H)+[345]+94+\sqrt{n}+\sqrt[3]{n}\right)
\end{equation*}
up to $n_0=25000$, adding the $[1083]$ only while $n\ge 1536$, we find that the maximum is $F(251)<1.27949$. (We have $R_\mathcal{E}(251)=0.82$ and $r_{251}(H)=0$ for all $H\in\mathcal{X}$.) Afterwards, the function
\begin{equation*}
\bar{F}(n)=\bar{R}_\mathcal{E}(n)+0.03389+\frac{1083+94}{n}+\frac{1}{\sqrt{n}}+\frac{1}{n^{2/3}}
\end{equation*}
is decreasing, $F(n)<\bar{F}(n)$, and $\bar{F}(n_0)<F(251)$.

As $F(n)$ is an upper bound to $Q_\mathcal{E}(n)$ for $n>250$, we see that the bound stated in the theorem holds also for $n>250$. This proves the claim.
\end{proof}

\section{Main result}\label{section:main_theorem}

The Main Theorem \ref{theorem:main_theorem} is proved along the same lines as Theorem 6.3 in~\cite{GuralnickLarsenTiep}. Let $G$ be an almost simple group having as its socle $G_0=G_0(V,\form)$, a classical simple group preserving the form $\form$ on the $n$-dimensional $\F_q$-space $V$. (The form is either trivial, or a non-degenerate alternating, quadratic or hermitian form, and the action is projective). With this notation, we have $G_0\unlhd G\le\aut(G_0)$, and apart from certain known exceptions, the automorphism group is $\PGamma(V,\form)$, the group of projective semilinear transformations preserving $\form$ up to scalar multiplication.
The exceptions occur when $G_0$ is one of $\PSL_n(q)$, $\Sp_4(2^k)$ or $\POmega_8^+(q)$, with additional (graph) automorphisms of order 2, 2 and 3, respectively. (See \cite[Th.~2.1.4]{BlueBook}.)

Let $M$ be a maximal subgroup of $G$ not containing $G_0$. It is known that the intersection $M\cap G_0$ is non-trivial (see e.g.\ \cite[p.~395]{LPSONanScott}), so we must have $M=N_G(M\cap G_0)$. It then follows that the intersection of $M\cap G_0$ completely determines $M$, and moreover, two maximal subgroups $M_1$ and $M_2$ are conjugate in $G$ precisely when the intersections $M\cap G_0$ and $M\cap G_1$ are. Thus, to find an upper bound for the number of $G$-conjugacy classes of maximal subgroups $M$ not containing $G_0$, it suffices to bound the number of $G_0$-classes of the intersections $M\cap G_0$.

Let $\Delta$ stand for the group of similarities of $V$ preserving $\form$, and let $\PDelta$ be the corresponding projective group. It is often the number of $\PDelta$-classes of $G_0\cap M$ which is the easiest to bound. When this is the case, we need to insert a factor of $\sigma=[\PDelta:G_0]$ to account for the splitting of $\PDelta$-classes under $G_0$. In the linear and unitary groups, $\sigma$ is at most the dimension of $V$, otherwise it is at most 8 (see Table~2.1.D in \cite{BlueBook}).

Assume now that $\aut(G_0)=\PGamma(V,\form)$ (the general case). To estimate the number of $G_0$-classes of maximal subgroups, we make use Aschbacher's Theorem on maximal subgroups of classical groups (\cite{Aschbacher}, see also \cite{BlueBook}). By that theorem, $M$ either belongs to one of the so-called \emph{geometrical families} $\geomclass_1$--$\geomclass_8$ of maximal subgroups, or in a further family consisting of almost simple groups. This additional family will be called $\sclass$.

Let us look at $\sclass$ more closely. By Aschbacher's Theorem, any maximal subgroup in this family is itself an almost simple group whose non-abelian simple socle acts absolutely irreducibly on $V$. Let $M$ be one of these maximal subgroups and let $M_0$ denote its socle. Then we have $M_0\le M\cap G_0\lhd M$ and $N_G(M_0)=M$. As before, it follows that it is enough to consider the $G_0$-conjugacy classes of the socles.

The socle $M_0$, via its action on $V$, corresponds to an absolutely irreducible modular projective representation of a simple group isomorphic to it. The projective representation of the simple group in turn lifts to a linear representation of the universal covering group $\tilde{M_0}$, and equivalent linear representations correspond to subgroups conjugate under $\GL(V)$. Furthermore, the conjugating element will always reside in $\Delta$ (see e.g.\ \cite[Corollary~2.10.4]{BlueBook}).

The number of cross-characteristic representations of groups of Lie type was bounded in Corollary~\ref{corollary:final_constant}. Assume now that the characteristic $\ell$ of the representation divides $q$. In this case, we will apply Steinberg's Tensor Product Theorem, which states that an irreducible representation of a simply-connected quasisimple group of Lie type is a tensor product of Frobenius twists of so-called restricted representations (\cite[Theorem~5.4.5]{BlueBook}). For the exceptional covering groups, which are not of the simply-connected type, we note that the $\ell$-part of the Schur multiplier of $M_0$ belongs to the kernel of the representation. From a list of Schur multipliers of groups of Lie type, we find that dividing out the said $\ell$-part, the remaining part is just the regular Schur multiplier. It follows that any representation of an exceptional covering group is a lift of a representation of the simply-connected group.

Imitating \cite{GuralnickLarsenTiep}, we now divide the non-geometrical subgroups into five subfamilies according to the nature of the socle $M_0$
and the representation:
\begin{itemize}
\item[$\sclass_1$] The socle is an alternating group.
\item[$\sclass_2$] The socle is a sporadic group.
\item[$\sclass_3$] The socle is a group of Lie type with defining characteristic not dividing~$q$.
\item[$\sclass_4$] The socle is a group of Lie type with defining characteristic dividing $q$, and the representation is not restricted.
\item[$\sclass_5$] The socle is a group of Lie type with defining characteristic dividing $q$, and the representation is restricted.
\end{itemize}
For each subfamily $\sclass_i$, we want to bound the number of inequivalent irreducible projective representations of simple groups in that subfamily from above, as this gives a bound to the number of $\PDelta$-conjugacy classes of maximal subgroups appearing in $\sclass_i$. As already mentioned, these projective representations of the simple groups can also be viewed as linear representations of their covering groups.

Next, we present some lemmata that will be useful for the coming calculations. The first is a result of F.~Lübeck.

\begin{lemma}[\cite{LuebeckSameChar}, Theorems 4.4 and 5.1]\label{lemma:lubeck_same_char}
Let $H$ be the universal covering group of a simple classical group, and let $\rho$ be a restricted, absolutely irreducible representation of $H$ in the defining characteristic. If $H$ is of type $A_r$ or $\twistA_r$ for some $r$, then either $\deg(\rho)>r^3/8$, or otherwise $\rho$ is the natural representation or its dual, or one of at most 3 dual pairs of representations. If $H$ is of another classical type of rank $r$, then one of the following holds:
\begin{enumerate}
\item[(a)] If $r>11$, then either $\dim(\rho)>r^3$, or otherwise $\rho$ is the natural representation or one of at most 2 representations.
\item[(b)] If $r\le 11$, the either $\dim(\rho)>r^3$, or otherwise $\rho$ is the natural representation or $\dim(\rho)$ is one of at most 5 possibilities.
\end{enumerate}
\end{lemma}

The second lemma gives some information on modular representations of alternating groups and their covers.

\begin{lemma}\label{lemma:num_alts} Let $\altcover(d)$ denote the universal covering group of a simple alternating group $\alt(d)$. The number of $d$, such that there exists an irreducible modular representation of $\altcover(d)$ of dimension $n$, is less than $2\sqrt{n}+4.33\ln(4n)+14$.
\end{lemma}

\begin{proof} For dimensions $n\le 250$, Hiss and Malle have listed all the representations of finite quasisimple groups in \cite{HissMalle250}, so the claim can be readily checked in this case.

Assume that $n>250$. When an irreducible representation of the symmetric group $\symm(d)$ is restricted to $\alt(d)$, it may stay irreducible or split into two irreducible representations of half the dimension. Define $f(d)=\frac{1}{2}(d-1)(d-2)$. By \cite[Theorem~7]{JamesSymmetricModular}, there are only three pairs of conjugate $\ell$-regular partitions that correspond to non-trivial irreducible $\ell$-modular representations of $\symm(d)$ of dimension less than $f(d)$, provided that $d\ge 15$. Depending on $\ell$, there are two possibilities for the dimension of each of these representations (see~\cite[Appendix]{JamesSymmetricModular}). The dimensions are given by strictly increasing polynomials in $d$. Taking into account the possibility that some of these exceptional representations may split in $\alt(d)$, we can say that there are at most 12 possible values for the dimension of an irreducible representation of $\alt(d)$ below $f(d)/2$.

Now, assuming that $n$ is a dimension of an irreducible representation of $\alt(d)$, at least one of the following must hold:
\begin{enumerate}
\item[(1)] $n\ge f(d)/2$
\item[(2)] $d\ge 15$ and the value of $n$ is one of 12 possibilities
\item[(3)] $d<15$.
\end{enumerate}
On the other hand, having $d<15$ without $n\ge f(d)/2$ contradicts our assumption that $n>250$. Hence either (1) or (2) holds, so the number of possible values of $d$, for which there exists an $n$-dimensional irreducible modular representation of $\alt(d)$, is at most 12 plus the number of $d$ such that $f(d)/2\le n$. As $f$ is strictly increasing, the sum is at most $f^{-1}(2n)+12<2\sqrt{n}+14$.

Let us then look at faithful representations of $\altcover(d)$. We may assume that $d>9$, since otherwise all irreducible representations of $\altcover(d)$ have dimension less than 250. From \cite{WagnerProjAlt}, we find that the smallest dimension of a faithful irreducible representation of $\altcover(d)$ is at least $g(d)=2^{\lfloor(d-s-1)/2\rfloor}$, where $s$ is the number of 1's in the binary representation of $d$. Making an estimate $g(d)>0.25\cdot 1.26^d$, we find that the number of $d$, such that there exists a faithful $n$-dimensional irreducible representation of $\altcover(d)$, is less than $4.33\ln(4n)$. The claim follows.
\end{proof}

The final lemma deals with the exceptional automorphisms.

\begin{lemma}\label{lemma:exceptional_automorphisms} Let $G$ and $G_0$ be as above.
\begin{itemize}
\item[(a)] Suppose $G_0=\Sp_4(2^k)$ and $G\not\le\PGamma(V,\form)$. The number of $G$-conjugacy classes of maximal subgroups $M$ of $G$, such that $G_0\not\le M$ and $M$ does not appear in Aschbacher's classes $\geomclass_5$ or $\sclass$, is at most $5$.
\item[(b)] Suppose $G_0=\POmega_8^+(q)$. The number of $G$-conjugacy classes of maximal subgroups $M$ of $G$, such that $G_0\not\le M$ and $M$ does not appear in class $\geomclass_5$, is at most $44$.
\end{itemize}
\end{lemma}

\begin{proof}
(a) In \cite[Section 14]{Aschbacher}, a new system of families $\geomclass_1'$--$\geomclass_5'$ is given, which contains all maximal subgroups not included in $\sclass$. Moreover, $\geomclass_4'$ is the same family as the old family $\geomclass_5$. It is then proved that there is only one $\aut(G_0)$-con\-ju\-ga\-cy class in each of $\geomclass'_1$, $\geomclass'_3$ and $\geomclass'_5$. In $\geomclass'_2$, there are at most two $\aut(G_0)$-classes. In fact, the proof also shows that the classes do not split under $G$. This gives the result.

(b) In this case, the number of conjugacy classes can be counted from \cite[Section 1.5]{KleidmanOmega8}, where P.~Kleidman lists all conjugacy classes of maximal subgroups of $G$.
\end{proof}

Now we are in a position to prove the main result of this paper. We retain the notation presented in the beginning of this section.

\begin{proof}[Proof of Theorem \ref{theorem:main_theorem}]
We assume that $G_0$ is not a group mentioned in parts (a) and (b) of Lemma \ref{lemma:exceptional_automorphisms}. These cases are dealt with at the end of the proof. The only other groups $G$ containing exceptional automorphisms are the automorphism groups of $G_0=\PSL_n(q)$. Tor these groups, Aschbacher presents another family, $\geomclass_1'$, that contains the additional maximal subgroups. We will follow the example of \cite{BlueBook} and include the new family $\geomclass_1'$ in $\geomclass_1$ in this case.

In accordance with the discussion above, we want to bound the number of $G_0$-conjugacy classes of groups $M\cap G_0$, where $M$ is a maximal subgroup of $G$. We will examine each Aschbacher family separately, bounding the number of conjugacy classes with $M$ belonging to that family. Afterwards, all the bounds will be added together. Let us first look at the geometrical families.

\textbf{Subfield family.} Family $\geomclass_5$ contains groups defined over a subfield of $\F_q$ of prime index. By \cite[Lemma 2.1]{GKS}, the number of $\PDelta$-conjugacy classes of maximal subgroups (intersected with $G_0$) is at most $\log_2\log_2 q+1$. We need to add a factor of $\sigma$ to account for the splitting, so that the total number of the $G_0$-conjugacy classes of maximal subgroups of the subfield family becomes at most $\sigma(\log_2\log_2 q+1)$.

\textbf{Other geometrical families.} We refer again to \cite[Lemma 2.1]{GKS}. There, an upper bound to the number of $\Delta$-classes of maximal subgroups of non-subfield type is given as
\begin{equation*}
\frac{3n}{2}+4d(n)+\pi(n)+3\log_2 n+8,
\end{equation*}
where $d(n)$ is the number of divisors of $n$, and $\pi(n)$ the number of prime divisors. (This also includes the family $\geomclass_1'$.) Using the estimate $\pi(n)\le d(n)\le\log_2 n$ and multiplying by $\sigma$ to account for splitting under $G_0$, we find that the number of $G_0$-classes of maximal subgroups is at most $\sigma(\frac{3}{2}n+8\log_2 n+8)$ in this case.

\medskip

Let us then turn to the exceptional family $\sclass$. We will deal with each subfamily $\sclass_i$ separately.

\textbf{Case $\sclass_1$.} By Lemma \ref{lemma:num_alts} above, the number of non-isomorphic $\alt(d)$ that can yield an $n$-dimensional irreducible projective representation is at most $2\sqrt{n}+4.33\ln(4n)+14$.
On the other hand, it was shown in Theorem 1.1.(ii) of \cite{GuralnickLarsenTiep} that the number $R_n(\alt(d))$ of $n$-dimensional irreducible projective representations of $\alt(d)$ is less than $n^{2.5}$. Hence,
\begin{equation*}
\sum_d R_n(\alt(d))<2n^3+4.33n^{2.5}\ln(4n)+14n^{2.5}.
\end{equation*}
As explained above, the number of inequivalent irreducible projective representations can be used as an upper bound to $\PDelta$-conjugacy classes of maximal subgroups in the family $\sclass$. We still need to take into account the splitting of $\PDelta$-classes under $G_0$ by inserting a factor of $\sigma$. Hence, the number of conjugacy classes of maximal subgroups of this type is bounded by $\sigma(2n^3+4.33n^{2.5}\ln(4n)+14n^{2.5})$.

\textbf{Case $\mathcal{S}_2$.} We divide the sporadic groups into two sets. The first one contains those for which complete information on representation degrees is available in the \atlas es. We call these the \emph{small} sporadic groups. For small sporadic groups, we list the maximal multiplicity of any representation degree of the universal covering group in Table~\ref{table:sporadic_multiplicities}. Adding these up will give an upper bound for the total number of $n$-dimensional representations of small sporadic groups, for any fixed $n$.

\begin{table}[hbt]
\[
\begin{array}{lcccccccccc}
\toprule
\text{group} & M_{11} & M_{12} & M_{22} & M_{23} & M_{24} & J_1 & J_2 & J_3 & HS & McL \\
\midrule
\text{multiplicity} & 3 & 3 & 8 & 3 & 3 & 4 & 2 & 6 & 3 & 6 \\
\bottomrule
\end{array}
\]
\caption{Maximal multiplicities of representation degrees (over all characteristics)
for the universal covering groups of small sporadic groups}\label{table:sporadic_multiplicities}
\end{table}

The other set contains the \emph{big} sporadic groups. For these, we shall make a crude approximation based solely on conjugacy class numbers of the universal covering groups. These are listed in Table~\ref{table:sporadic_class_numbers}. Adding them up will give a bound for the total number of representations.

The smallest representation degree of a big sporadic group is at least 12 (the group is the Suzuki group, see \cite[Prop.\ 5.3.8]{BlueBook}). In conclusion, we can bound the number of $\PDelta$-classes in subfamily $\sclass_2$ by 41, when $n\le 12$, and by 1988 afterwards. This needs to be multiplied by $\sigma$ to account for splitting.

\begin{table}[hbt]
\begin{align*}
& \begin{array}{lcccccccc}
\toprule
\text{group} & J_4 & He & Ru & Suz & O\text{'}N & Co_1 & Co_2 & Co_3 \\
\midrule
\text{class number} & 62 & 33 & 61 & 210 & 80 & 167 & 60 & 42 \\
\bottomrule
\end{array} \\
& \begin{array}{lcccccccc}
\toprule
\text{group} & Fi_{22} & Fi_{23} & Fi_{24}' & HN & Ly & Th & BM & M \\
\midrule
\text{class number} & 282 & 98 & 256 & 54 & 53 & 48 & 247 & 194 \\
\bottomrule
\end{array}
\end{align*}
\caption{Class numbers of the universal covering groups of big sporadic groups}\label{table:sporadic_class_numbers}
\end{table}

\textbf{Case $\mathcal{S}_3$.} This subfamily corresponds to cross-characteristic representations of finite quasisimple groups of Lie type. By Corollary~\ref{corollary:final_constant}, the total number of equivalence classes of such representations is less than $15n$. As before, we need to add a factor of $\sigma$ to count for the splitting of the corresponding $\PDelta$-classes.

\textbf{Case $\mathcal{S}_4$.} By Steinberg's Tensor Product Theorem, the representation is a non-trivial tensor product of twisted restricted representations. Denote $G_0=\mathrm{P}Cl_{y^s}(q)$, a classical group of dimension $y^s$ over the field $\F_q$. Now, from Corollary~6 of \cite{SeitzNonRestricted} it follows that $M$ normalises a classical subgroup of $G_0$ of the form $\mathrm{P}Cl_y(q^s)$. These subgroups are completely described in \cite{Schaffer}, and we can read from Table~1B of that work that the number of subgroups of this type is at most $a+2$, where $a$ is the number of ways to write the dimension $n$ as a power $y^s$. Thus, the number of subgroups of the form $\mathrm{P}Cl_y(q^s)$ is bounded above by $\log_2\log_2 n+2$. This number also bounds the number of conjugacy classes of $M$.

\textbf{Case $\mathcal{S}_5$.} In this case, the socle of $M$ is the image in $G_0$ of a projective representation of a finite simple group $H=H_{r'}(q')$ of Lie type. The homomorphism is injective, so we can identify $H$ with the socle. The representation is characterised by a highest weight $\lambda$, as explained e.g.\ in \cite[§5.4]{BlueBook}. We write $V(\lambda)$ for the module of the representation.

Assume first that $G_0=\PSL_n(q)$. We know that $H$ does not preserve a non-degenerate alternating, Hermitian or quadratic form, for otherwise $M$ would belong to Aschbacher's family $\geomclass_8$ which consists of classical subgroups of $G$.

If $H=\PSL_{r'+1}(q')$, we must have $q=q'$. Namely, if it were that $q'>q$, then Proposition~5.4.6 of \cite{BlueBook} would tell us that the representation is not restricted. On the other hand, we cannot have $q'<q$, because $M$ is not in the subfield family $\geomclass_5$.

By Lemma \ref{lemma:lubeck_same_char}, either $(r')^3/8<n$, or the representation belongs to one of at most 3 dual pairs of representations. In the former case, $r'$ can be one of at most $2n^{1/3}$ possibilities, and by \cite[Theorem~1.1.(i)]{GuralnickLarsenTiep}, $\PSL_{r'+1}(q)$ has at most $n^{3.8}$ restricted representations of degree $n$. This gives altogether at most
\[
2n^{1/3+3.8}+3
\]
conjugacy classes of subgroups under $\PDelta$.

If $H=\PSU_{r'+1}(q')$, there is no corresponding subgroup of $G$. Namely, it follows from Theorems 5.4.2 and 5.4.3 presented in \cite{BlueBook}, that the dual module $V(\lambda)^*$ is isomorphic to either $V(\lambda)$ or  $V(\lambda)^\psi$, where $\psi$ is the involutory automorphism of $\F_q=\F_{q'^2}$.
In the first case the group would fix a non-degenerate bilinear form, and in the latter case it would fix a Hermitian form (see \cite[Lemma 2.10.15]{BlueBook}). Similarly, $H$ cannot be of type $B$ or $C$ as groups of these types have only self-dual representations.

If $H$ is of one of the two remaining orthogonal types, there are at most 2 possibilities for $q'$ (either $q'=q$ or possibly $q'=q^{1/2}$). In this case, Lemma~\ref{lemma:lubeck_same_char} tells us that either $(r')^8<n$, or the representation is one of at most 5 different possibilities. In the former case, there are at most $n^{1/3}$ possibilities for $r'$, and by \cite[Theorem~1.1.(i)]{GuralnickLarsenTiep}, the number of restricted $n$-dimensional representations of $H_{r'}(q')$ is at most $n^{2.5}$. Thus the number of $\PDelta$-conjugacy classes of maximal subgroups of types $D$ and $\twistD$ is at most $2(2n^{1/3+2.5}+2)$.

For the exceptional types, only $E_6$ and $\twistE_6$ have other than self-dual representations (\cite[Proposition~5.4.3]{BlueBook}). There are at most two possibilities for $q'$ with the twisted type, and the number of restricted $n$-dimensional representations of each group is at most $n^{2.5}$ by \cite[Theorem~1.1.(i)]{GuralnickLarsenTiep}. Hence, the number of $\PDelta$-conjugacy classes is at most $3n^{2.5}$.

As a conclusion, the number of $\PDelta$-conjugacy classes has been bounded in the case $G_0=\PSL_n(q)$. To bound the number of $G_0$-classes, we multiply the original bound by $\sigma$, which in this case is at most $n$. Hence, we get the following bound for the number of $G_0$-classes:
\begin{equation*}
b_1(n)=2n^{5.14}+4n^{3.84}+3n^{3.5}+13n.
\end{equation*}

Similar analysis can be performed when $G_0$ is of any other classical type. In the case $G_0=\PSU_n(q)$, the main difference is that it is not possible to have $H=\PSL_{r'+1}(q')$. This is because we would need the dual module $V(\lambda)^*$ to be isomorphic to $V(\lambda)^{\psi}$, where $\psi$ is the involutory automorphism of $\F_{q^2}$, but this would lead to an impossible equation on the weights, since $\lambda$ is restricted. Instead, we may well have $H=\PSU_{r'+1}(q')$, and the estimates lead to the same bound $b_1(n)$ as in the previous case.

On the other hand, when $G_0$ is neither linear nor unitary, we have $\sigma\le 8$, and we can make the estimate simpler by assuming that $H$ can be of any Lie type and saying that there are always at most two possibilities for the value of~$q'$. The full details are omitted, but the intermediate results are gathered in the following table:
\[
\begin{array}{ccc}
\toprule
H & \text{no.\ of Lie types} & \text{bound} \\
\otoprule
\PSL_{r'}(q') & 1 & 4n^{1/3+3.8}+3 \\
\PSU_{r'}(q') & 1 & 4n^{1/3+3.8}+3 \\
\text{other classical} & 4 & 2n^{1/3+2.5}+5 \\
\text{exceptional} & 10 & 2n^{2.5} \\
\bottomrule
\end{array}
\]
Multiplied by 8, the bound becomes
\begin{equation*}
b_2(n)=64n^{4.14}+64n^{2.84}+160n^{2.5}+208.
\end{equation*}

For $n\ge 32$, we have $b_1(n)>b_2(n)$. When $n<32$, we can use Lemma~\ref{lemma:lubeck_same_char} together with additional data given in \cite{LuebeckSameChar} to bound the number of representations in the same way as above. Firstly, if $H_{r'}(q')$ is of type $A$ or $\twistA$, we must have $r'\le 6$, and we can check from the tables in \cite[Appendix~A]{LuebeckSameChar} that there are at most 5 representations of any particular dimension for both types (counting dual representations only once). For~$\twistA$, there are two possibilities for $q$, so that we get at most 15 representations of the same degree.

Similarly, for the other classical types, we have $r'$ at most 3, so the types $D$ can $\twistD$ can be left out. We get at most 4 representations of the same degree. For the exceptional types, we get at most 8 representations of the same degree, of which 4 come from types $E_6$ and $\twistE_6$ and 4 from $G_2$ and $\twistG_2$. Saying, for simplicity, that $q'$ can have at most 2 values for each exceptional type, we get 16 representations. Thus, there are altogether at most 35 restricted representations of any degree $n$ less than 32, so there are at most $35\sigma$ conjugacy classes under $G_0$, when $n<32$.

\medskip

\textbf{Conclusion.} Continue assuming that the automorphism group of $G$ is of generic type (that is, $\aut(G)=\mathrm{P\Gamma}(V,\form)$). Suppose first that $n\ge 32$, so that $\sigma\le n$ for all types of $G$. Using this estimate for $\sigma$, we collect into Table~\ref{table:final_bounds} the partial results obtained for each Aschbacher family $\geomclass_i$ and subfamily $\sclass_i$. Adding up the partial results, it is easy to check computationally that the bound given in the statement of the theorem holds for $n\ge 32$.

\begin{table}[hbt]
\[
\begin{array}{rl}
\toprule
\text{family} & \text{bound} \\
\otoprule
\text{subfield family $\geomclass_5$} & n(\log_2\log_2 q+1) \\
\text{other geometrical families} & \frac{3}{2}n^2+8n\log_2 n+8n \\
\midrule
\sclass_1 & 2n^4+4.33n^{3.5}\ln(4n)+14n^{3.5} \\
\sclass_2 & 1988n \\
\sclass_3 & 15n^2 \\
\sclass_4 & \log_2\log_2 n+2 \\
\sclass_5 & 2n^{5.14}+4n^{3.84}+3n^{3.5}+13n \\
\bottomrule
\end{array}
\]
\caption{Upper bounds for the number of conjugacy classes of different types of maximal subgroups of classical groups}\label{table:final_bounds}
\end{table}

For smaller values of $n$, we apply the other bound obtained for $\sclass_5$, namely $35\sigma$. For $n\ge 8$, we still have $\sigma\le n$, so we may substitute $35n$ in the place of the expressing in the last row of Table~\ref{table:final_bounds}. Adding up, we see that the bound given in the statement of the theorem holds for $8\le n<32$. Finally, P.~Kleidman has in his PhD thesis (\cite{KleidmanThesis}) determined all maximal subgroups of classical groups in dimension at most 11, and from this we can read that the bound given in the statement holds even for $n<8$.

It only remains to check the groups with exceptional automorphisms. For this, we use Lemma~\ref{lemma:exceptional_automorphisms}. The number of conjugacy classes of maximal subgroups of subfield type (family~$\geomclass_5$) is bounded by $n(\log_2\log_2 q+1)$, just like in the general case. Apart from these classes, the lemma gives 5 additional classes for $G_0=\Sp_4(2^k)$, and 44 classes for $G_0=\POmega_8^+(q)$. Evidently, the statement of the theorem holds in this case as well. This concludes the proof.
\end{proof}

\bibliographystyle{plain}
\bibliography{elsevierManuscript}

\begin{thebibliography}{10}

\bibitem{Aschbacher}
M.~Aschbacher.
\newblock On the maximal subgroups of the finite classical groups.
\newblock {\em Invent. Math.}, 76(3):469--514, 1984.

\bibitem{BurkhardtPSL}
R.~Burkhardt.
\newblock Die {Z}erlegungsmatrizen der {G}ruppen {${\rm PSL}(2,p^{f})$}.
\newblock {\em J. Algebra}, 40(1):75--96, 1976.

\bibitem{BurkhardtSuz}
R.~Burkhardt.
\newblock \"{U}ber die {Z}erlegungszahlen der {S}uzukigruppen {${\rm Sz}(q)$}.
\newblock {\em J. Algebra}, 59(2):421--433, 1979.

\bibitem{Atlas}
J.~H. Conway, R.~T. Curtis, S.~P. Norton, R.~A. Parker, and R.~A. Wilson.
\newblock {\em Atlas of finite groups}.
\newblock Oxford University Press, Eynsham, 1985.
\newblock Maximal subgroups and ordinary characters for simple groups, With
  computational assistance from J. G. Thackray.

\bibitem{Dornhoff}
Larry Dornhoff.
\newblock {\em Group representation theory. {P}art {B}: {M}odular
  representation theory}.
\newblock Marcel Dekker Inc., New York, 1972.
\newblock Pure and Applied Mathematics, 7.

\bibitem{GuralnickConjugacy}
Jason Fulman and Robert Guralnick.
\newblock Bounds on the number and sizes of conjugacy classes in finite
  {C}hevalley groups with applications to derangements.
\newblock {\em Trans. Amer. Math. Soc.}, 364(6):3023--3070, 2012.

\bibitem{GKS}
Robert~M. Guralnick, William~M. Kantor, and Jan Saxl.
\newblock The probability of generating a classical group.
\newblock {\em Comm. Algebra}, 22(4):1395--1402, 1994.

\bibitem{GuralnickLarsenTiep}
Robert~M. Guralnick, Michael Larsen, and Pham~Huu Tiep.
\newblock Representation growth in positive characteristic and conjugacy
  classes of maximal subgroups.
\newblock {\em Duke Math. J.}, 161(1):107--137, 2012.

\bibitem{GMSTUnitarySymplectic}
Robert~M. Guralnick, Kay Magaard, Jan Saxl, and Pham~Huu Tiep.
\newblock Cross characteristic representations of symplectic and unitary
  groups.
\newblock {\em J. Algebra}, 257(2):291--347, 2002.

\bibitem{GuralnickTiepLinear}
Robert~M. Guralnick and Pham~Huu Tiep.
\newblock Low-dimensional representations of special linear groups in cross
  characteristics.
\newblock {\em Proc. London Math. Soc. (3)}, 78(1):116--138, 1999.

\bibitem{GuralnickTiepSymplecticEven}
Robert~M. Guralnick and Pham~Huu Tiep.
\newblock Cross characteristic representations of even characteristic
  symplectic groups.
\newblock {\em Trans. Amer. Math. Soc.}, 356(12):4969--5023 (electronic), 2004.

\bibitem{HissHabilitation}
Gerhard Hiss.
\newblock {Z}er\-le\-gungs\-zahlen end\-licher {G}rup\-pen vom {L}ie-{T}yp in
  nicht-de\-fi\-nie\-render {C}ha\-rak\-te\-ris\-tik.
\newblock Ha\-bi\-li\-ta\-tions\-schrift, 1990.
\newblock available from
  \url{http://www.math.rwth-aachen.de/~Gerhard.Hiss/Preprints/} (accessed 13
  Oct 2011).

\bibitem{HissMalle250}
Gerhard Hiss and Gunter Malle.
\newblock Low-dimensional representations of quasi-simple groups.
\newblock {\em LMS J. Comput. Math.}, 4:22--63 (electronic), 2001.

\bibitem{HissMalleUnitary}
Gerhard Hiss and Gunter Malle.
\newblock Low-dimensional representations of special unitary groups.
\newblock {\em J. Algebra}, 236(2):745--767, 2001.

\bibitem{HissMalle250Corrigenda}
Gerhard Hiss and Gunter Malle.
\newblock Corrigenda: ``{L}ow-dimensional representations of quasi-simple
  groups'' [{LMS} {J}. {C}omput.\ {M}ath.\ {\bf 4} (2001), 22--63; {MR}1835851
  (2002b:20015)].
\newblock {\em LMS J. Comput. Math.}, 5:95--126 (electronic), 2002.

\bibitem{HoffmanMinimal}
Corneliu Hoffman.
\newblock Cross characteristic projective representations for some classical
  groups.
\newblock {\em J. Algebra}, 229(2):666--677, 2000.

\bibitem{HoffmanTypeE}
Corneliu Hoffman.
\newblock Projective representations for some exceptional finite groups of
  {L}ie type.
\newblock In {\em Modular representation theory of finite groups
  ({C}harlottesville, {VA}, 1998)}, pages 223--230. de Gruyter, Berlin, 2001.

\bibitem{JamesSymmetricModular}
G.~D. James.
\newblock On the minimal dimensions of irreducible representations of symmetric
  groups.
\newblock {\em Math. Proc. Cambridge Philos. Soc.}, 94(3):417--424, 1983.

\bibitem{ModularAtlas}
Christoph Jansen, Klaus Lux, Richard Parker, and Robert Wilson.
\newblock {\em An atlas of {B}rauer characters}, volume~11 of {\em London
  Mathematical Society Monographs. New Series}.
\newblock The Clarendon Press Oxford University Press, New York, 1995.
\newblock Appendix 2 by T. Breuer and S. Norton, Oxford Science Publications.

\bibitem{BlueBook}
Peter Kleidman and Martin Liebeck.
\newblock {\em The subgroup structure of the finite classical groups}, volume
  129 of {\em London Mathematical Society Lecture Note Series}.
\newblock Cambridge University Press, Cambridge, 1990.

\bibitem{KleidmanOmega8}
Peter~B. Kleidman.
\newblock The maximal subgroups of the finite {$8$}-dimensional orthogonal
  groups {$P\Omega^+_8(q)$} and of their automorphism groups.
\newblock {\em J. Algebra}, 110(1):173--242, 1987.

\bibitem{KleidmanThesis}
Peter~B. Kleidman.
\newblock {\em The Subgroup Structure of Some Finite Simple Groups}.
\newblock PhD thesis, University of Cambridge, 1987.

\bibitem{LandazuriMinimal}
Vicente Landazuri and Gary~M. Seitz.
\newblock On the minimal degrees of projective representations of the finite
  {C}hevalley groups.
\newblock {\em J. Algebra}, 32:418--443, 1974.

\bibitem{LandrockMichler}
Peter Landrock and Gerhard~O. Michler.
\newblock Principal {$2$}-blocks of the simple groups of {R}ee type.
\newblock {\em Trans. Amer. Math. Soc.}, 260(1):83--111, 1980.

\bibitem{LPSONanScott}
Martin~W. Liebeck, Cheryl~E. Praeger, and Jan Saxl.
\newblock On the {O}'{N}an-{S}cott theorem for finite primitive permutation
  groups.
\newblock {\em J. Austral. Math. Soc. Ser. A}, 44(3):389--396, 1988.

\bibitem{Luebeck2E6}
Frank L{\"u}beck.
\newblock Conjugacy classes and character degrees of
  {$^2\!E_6(2)_{\mathrm{sc}}$}.
\newblock available from
  \url{http://www.math.rwth-aachen.de/~Frank.Luebeck/chev/2E62.html}, accessed
  16th May 2012.

\bibitem{LuebeckSameChar}
Frank L{\"u}beck.
\newblock Small degree representations of finite {C}hevalley groups in defining
  characteristic.
\newblock {\em LMS J. Comput. Math.}, 4:135--169 (electronic), 2001.

\bibitem{LuebeckPolys}
Frank Lübeck.
\newblock Character degrees and their multiplicities for some groups of lie
  type of rank $<$ 9.
\newblock available from
  \url{http://www.math.rwth-aachen.de/~Frank.Luebeck/chev/DegMult/index.html},
  accessed 13 Oct 2011.

\bibitem{MMT}
Kay Magaard, Gunter Malle, and Pham~Huu Tiep.
\newblock Irreducibility of tensor squares, symmetric squares and alternating
  squares.
\newblock {\em Pacific J. Math.}, 202(2):379--427, 2002.

\bibitem{Schaffer}
Mark Schaffer.
\newblock Twisted tensor product subgroups of finite classical groups.
\newblock {\em Comm. Algebra}, 27(10):5097--5166, 1999.

\bibitem{SeitzNonRestricted}
Gary~M. Seitz.
\newblock Representations and maximal subgroups of finite groups of {L}ie type.
\newblock {\em Geom. Dedicata}, 25(1-3):391--406, 1988.
\newblock Geometries and groups (Noordwijkerhout, 1986).

\bibitem{SeitzZalesskiiMinimal}
Gary~M. Seitz and Alexander~E. Zalesskii.
\newblock On the minimal degrees of projective representations of the finite
  {C}hevalley groups. {II}.
\newblock {\em J. Algebra}, 158(1):233--243, 1993.

\bibitem{SuzukiGroups}
Michio Suzuki.
\newblock On a class of doubly transitive groups.
\newblock {\em Ann. of Math. (2)}, 75:105--145, 1962.

\bibitem{TiepMinimal}
Pham~Huu Tiep.
\newblock Low dimensional representations of finite quasisimple groups.
\newblock In {\em Groups, combinatorics \& geometry ({D}urham, 2001)}, pages
  277--294. World Sci. Publ., River Edge, NJ, 2003.

\bibitem{WagnerProjAlt}
Ascher Wagner.
\newblock An observation on the degrees of projective representations of the
  symmetric and alternating group over an arbitrary field.
\newblock {\em Arch. Math. (Basel)}, 29(6):583--589, 1977.

\bibitem{WardReeGroups}
Harold~N. Ward.
\newblock On {R}ee's series of simple groups.
\newblock {\em Trans. Amer. Math. Soc.}, 121:62--89, 1966.

\end{thebibliography}

\end{document}